\documentclass{amsart}
\usepackage{float}
\usepackage{graphicx} 
\usepackage{JohnsonLe}

\usepackage{todonotes}

\makeatletter
\renewenvironment{proof}[1][\proofname]{%
   \par\pushQED{\qed}\normalfont%
   \topsep6\p@\@plus6\p@\relax
   \trivlist\item[\hskip\labelsep\bfseries#1\@addpunct{.}]%
   \ignorespaces
}{%
   \popQED\endtrivlist\@endpefalse
}

\usepackage[backend=biber,style=alphabetic,sorting=nyt,
    url=false,doi=false,isbn=false,
]{biblatex}
\title{Order-Preserving outer automorphisms of free and surface groups}
\author{Jonathan Johnson and Khanh Le}
\date{\today}

\usepackage{lineno}

\begin{document}
\begin{abstract}
We give a complete classification of finite subgroups of outer automorphisms that preserve bi-orders on non-abelian free groups and bi-orderable surface groups. We also give another new criterion for an outer automorphism of $F_n$ induced by the action of an $n$-strand braid to preserve a bi-order on $F_n$ in terms of the Burau representation of $B_n$. Using the new criterion, we produce examples of order-preserving braids whose underlying permutation is a full cycle, which answers in affirmative a question of Kin and Rolfsen. 
\end{abstract}

\maketitle

\section{Introduction}
A group $G$ is \emph{bi-orderable} if it admits a total order $\prec$ that is invariant under both left and right multiplication. That is, if $g \prec h$ in $G$, then $fg \prec fh$ and $gf \prec hf$ for any $f\in G$. We call this total order a \emph{bi-order} on $G$. Alternatively, a bi-order on a group $G$ can be given by specifying a subset $P$ of $G$ that satisfies the following properties:
\begin{enumerate}
    \item $G = P \sqcup \{1\} \sqcup P^{-1}$,
    \item $P \cdot P \subset P$,
    \item $gPg^{-1} \subset P$ for all $g \in G$.
\end{enumerate}
The set $P$ determines a total order $\prec$ on $G$ by 
\[
g \prec h \text{ if and only if } g^{-1}h \in P,
\]
and is called \emph{a positive cone of $G$.} If the total order $\prec$ is known only to be left-invariant then $\prec$ is called a \emph{left-order} on $G$. The group $G$ is then said to be \emph{left-orderable.} The theory of orderable groups has attracted recent attention due to their connections with low-dimensional topology in the context of the L-space conjecture which concerns left-orderability of the fundamental group of rational homology 3-spheres, see \cite{BoyerGordonWatson13, Juhasz15}. 

Bi-orderability of the fundamental group of 3-manifolds has also been considered since the formulation of the L-space conjecture. In particular, Boyer, Rolfsen and Wiest gave a classification of Seifert fibered spaces with bi-orderable fundamental groups \cite{BoyerRolfsenWiest05}. Perron and Rolfsen give a sufficient criterion for the fundamental groups of a fibered 3-manifolds to be bi-orderable in terms of the monodromy of the fiber \cite{PerronRolfsen06, PerronRolfsen03}. Continuing on the theme of studying bi-orderability of 3-manifold groups, Kin and Rolfsen initiated the study of bi-orderability of the complement of a braided link, which is the mapping torus of a homeomorphism of an $n$-punctured disk induced by a braid in $B_n$ \cite{KinRolfsen18}. These are natural examples of fibered 3-manifold groups whose monodromy satisfies a necessary condition by \cite{ClayRolfsen12} for the 3-manifold group to be bi-orderable. In their article, Kin and Rolfsen classified the complement of braided link with bi-orderable fundamental group among those with finite-order monodromy \cite[Theorem 4.10]{KinRolfsen18} and gave more examples of ones with non-bi-orderable fundamental group. In this paper, we continue the investigation into bi-orderability of groups arising in low-dimensional topology by looking at mapping torus of free and surface groups.

Many groups are known to be bi-orderable: torsion-free abelian groups \cite{Levi42}, free groups \cite{Magnus35} and fundamental groups of closed surfaces except for the real projective plane and the Klein bottle \cite{RolfsenWiest01}.
New examples of orderable groups can be constructed from known examples using group extensions. Before stating the construction, we observe that given a bi-orderable group $G$, the group $\Out(G)$ acts on the set of all bi-orders of $G$, $\BiO(G)$, via:
\[
[\varphi]\cdot P = \varphi(P)
\]
where $\varphi \in \Aut(G)$ and $P$ a positive cone of a bi-order on $G$. The action is well-defined since positive cones of every bi-order of $G$ is invariant under conjugation by $G$. We also recall that given any short exact sequence of groups,
\[
1 \to K \to G \xrightarrow{\pi} Q \to 1,
\]
we have an associated action $\rho:Q \to \Out(K)$ given by 
\[
\rho(q)(k) = gkg^{-1} \text{ where } \pi(g) = q
\]
for any $q \in Q$, $k\in K$. Now we state a well-known construction of a bi-orderable group using group extensions:

\begin{prop}\label{prop:NewBOGroupViaGroupExtension}
Let $G$ be an extension of two bi-orderable groups $K$ and $Q$. In particular, we have the following short exact sequence of groups
\[
1 \to K \to G \to Q \to 1
\]
where $K$ and $Q$ are bi-orderable groups. Let $\rho:Q \to \Out(K)$ be the action associated to the short exact sequence. If there exists a bi-order $P \in \BiO(K)$ such that $\rho(Q)$ fixes $P$, then $G$ is bi-orderable. 
\end{prop} 
Due to \cref{prop:NewBOGroupViaGroupExtension}, we have the following concept. Given a bi-orderable group $K$, a subgroup of $\Out(K)$ is \emph{order-preserving} if it fixes a bi-order in $\BiO(K)$. It is natural to ask

\begin{quest}\label{quest:OPSubgroupOfOuter}
    Given a bi-orderable group $G$, which subgroups of $\Out(G)$ are order-preserving?
\end{quest}

The first result in our paper is a complete characterization of order-preserving finite subgroups of $\Out(G)$ when $G$ is either a non-abelian free group or the fundamental group of a closed surface group except for the real projective plane and the Klein bottle. In \cite{BoyerRolfsenWiest05}, the authors proved that all closed surfaces except for $\mathbb{RP}^2$ and the Klein bottle are precisely the closed surfaces with bi-orderable fundamental groups. For convenience, we refer to the set of fundamental groups of closed surfaces, other than the real projective plane and the Klein bottle, as \emph{bi-orderable surface groups}.

\begin{thm}\label{thm:FiniteOrderOP}
    Let $G$ be either a finitely generated non-abelian free group or a bi-orderable surface group. Let us denote by $\pi:\Aut(G) \to \Out(G)$ the canonical quotient map in the following short exact sequence:
    \[
    1 \to G \to \Aut(G) \to \Out(G) \to 1.
    \]
    Suppose that $H \leq \Out(G)$ is a finite subgroup. Then $H$ is order-preserving if and only if the full preimage $\pi^{-1}(H)$ under $\pi$ is torsion-free. 
\end{thm}

\cref{thm:FiniteOrderOP} is a generalization of Kin and Rolfsen's classification of order-preserving periodic braids\cite[Theorem 4.10]{KinRolfsen18}. Khramtsov gave an algorithm that, given a free-by-finite group, the algorithm decomposes the group as the fundamental group of a finite graph of finite groups \cite[Theorem 1]{Khramtsov95}. Using this algorithm and \cref{thm:FiniteOrderOP}, we have the following corollary:

\begin{cor}
\label{cor:OrderabilityFreeByFiniteIsAlgorthimicallyDeciable}
Let $G$ be a finitely generated non-abelian free group. There exists an algorithm that: given any finite subgroup $H$ of $\Out(G)$, the algorithm will out put YES if $H$ is order-preserving and NO otherwise.
\end{cor}
\noindent
We also remark that the method of proving \cref{thm:FiniteOrderOP} yields a construction that provides new examples of infinite-order order-preserving outer automorphism of finitely generated non-abelian free groups and bi-orderable surface groups. See \cref{cor:OPOuterAutomorphismCommutingWithFiniteOrder} for the construction.  

The second result in our paper gives a new criterion for an infinite-order outer automorphism of a finitely generated free group to be order-preserving. To motivate this result, we recall a criterion which guarantees that an automorphism is order-preserving, due to Perron and Rolfsen \cite[Theorem 2.6]{PerronRolfsen03}. Let $G$ be a rank-$n$ free group. Each automorphism $\varphi\in\Aut(G)$ induces a map in $\varphi_{\mathrm{ab}}\in\GL_{n}(\mathbb{Z})$ on the abelianization of $G$.
Composing with the inclusion map $\GL_{n}(\mathbb{Z})\subset\GL_{n}(\mathbb{C})$ produces a representation $\rho:\Aut(G)\to \GL_{n}(\mathbb{C})$.
Perron and Rolfsen proves the following sufficient condition for an element in $\Aut(G)$ to be order-preserving.

\begin{thm}[{\cite[Theorem 2.6]{PerronRolfsen03}}]
If for $\varphi\in\Aut(G)$ all the eigenvalues of $\rho(\varphi) = \varphi_{\mathrm{ab}}$ are real and positive, then $\varphi$ is order-preserving.
\end{thm}

We prove an analogous theorem for the Artin action of the braid group on $n$ strands on the free group of rank $n$ \cref{eq:ArtinAction} by formulating a positive eigenvalues condition for the reduced Burau representation for braid groups. We briefly recall the background on the reduced Burau representation. The braid group $B_n$ is isomorphic to the mapping class group of the $n$-punctured disk $D_n$ fixing the boundary component pointwise. The induced action on the fundamental group, $F$, of $D_n$ with a basepoint on the boundary is known as the Artin action of $B_n$ on the free group on $n$ generators. The total algebraic winding number around the punctures of $D_n$ defines a homomorphism $\sigma: F \to \mathbb{Z}$. The action of the mapping class group preserves the total algebraic winding number and lifts to the infinite cyclic cover $\widetilde{D}_n$ of $D_n$ that corresponds to $K = \ker(\sigma)$. The first homology group of $K$ naturally has the structure of a free $\mathbb{Z}[t^{\pm 1}]$-module. The action of the braid group on $H_1(K;\mathbb{Z})$ is a module action and defines the reduced Burau representation $\rho:B_n \to \GL_{n-1}(\mathbb{Z}[t^{\pm 1}])$, see \cite[Section 3.2.1, 3.3 and Remark 3.11]{kasselturaev08braid}. To formulate the condition of positive real eigenvalues, we extend the coefficients from $\mathbb{Z}[t^{\pm 1}]$ to the field of Puiseux series, denoted as $\mathbb{E}$, see \cref{subsec:puiseux} for the definition of $\mathbb{E}$. The field $\mathbb{E}$ is a real closed field and admits a unique ordering turning $\mathbb{E}$ into an ordered field. This ordering on $\mathbb{E}$ makes it possible to talk about positivity of elements. See \cref{subsec:OrderedRingAndFields} and \cref{subsec:puiseux} for precise definitions of these concepts. Our main result in this direction is 

\begin{thm}\label{thm:OPViaBurauRep}
    Let $\rho:B_n \to \GL_{n-1}(\mathbb{E})$ be the (reduced) Burau representation. If for a braid $\beta$ all eigenvalues of $\rho(\beta)$, counted with multiplicity,
    are positive in $\mathbb{E}$ equipped with the unique ordering, then $\beta$ is order-preserving.  
\end{thm}

The result can be directly generalized to the case of an arbitrary outer automorphism of a free or bi-orderable surface group that leaves a cohomology class invariant. As an application, we apply \cref{thm:OPViaBurauRep} to produce many new examples of order-preserving 3-strand braids. Before stating the result of our application, we recall the classification of 3-strand braids up to conjugation due to Murasugi:

\begin{thm}\cite{Murasugi74}
    \label{thm:MurasugiClassification}
    Let $\beta$ be a braid word in $B_3$. Then $\beta$ is conjugate to one of the following:
    \begin{enumerate}[label = (\alph*)]
        \item $\sigma_1\sigma_2^{-a_1}\dots\sigma_1\sigma_2^{-a_k}\Delta^{2d}$, where $k\geq 1$ and $a_i \geq 0$ with at least one $a_i \neq 0$ \label{murpsuedoanosov}
        \item $\sigma_1^k\Delta^{2d}$ for $k \in \mathbb{Z}$ \label{murreduce}
        \item $\sigma_1^k\sigma_2^{-1}\Delta^{2d}$ where $k \in \{-1,-2,-3\}$ \label{murperiodic}
    \end{enumerate}
    where $\Delta = (\sigma_1\sigma_2\sigma_1)^2$ and $d\in \mathbb{Z}$. 
\end{thm}

Kin and Rolfsen classified precisely which braids in the two families \ref{murreduce} and \ref{murperiodic} are order-preserving \cite[Proposition 4.5, Corollary 4.7, Theorem 4.10]{KinRolfsen18}.
The action of $\Delta^{2d}$ on the free group is an inner-automorphism. Therefore, $\Delta^{2d}\beta$ is order-preserving if and only if $\beta$ is order-preserving.
The classification of which 3-braids in family \ref{murpsuedoanosov} are order-preserving is an active open problem. The next two results make progress toward this classification problem. 

\begin{thm}
    \label{thm:EvenEvenBraidsAreOP}
    Let $\beta$ be a braid in $B_3$ that is conjugate to $\sigma_2^{-a_k}\sigma_1\dots\sigma_2^{-a_1}\sigma_1\Delta^{2d}$, where $a_i \geq 0$ with at least one $a_i \neq 0$. If both $k$ and $a_1+\dots+a_k$ are even, then $\beta$ is order-preserving. 
\end{thm}

\begin{cor}
    \label{cor:SquareOfThreeBraidsAreOP}
    Let $\beta \in B_3$ be any non-periodic 3-braid. Then $\beta^2$ is an order-preserving braid. 
\end{cor}

Prior to these results, the only known order-preserving braids in family \ref{murpsuedoanosov} are the ones provided by Cai, Clay and Rolfsen \cite{CaiClayRolfsen}. The first author, Scherich, and Turner show that 3-strand braids of the form $\sigma_1\sigma_2^{2k+1}$ for an integer $k$ are not order-preserving \cite[Theorem 7]{JohnsonScherichTurner2024}. By \cref{cor:SquareOfThreeBraidsAreOP}, the braids $(\sigma_1\sigma_2^{2k+1})^2$ are order-preserving. These are the first known examples of order-preserving braids whose permutation type is a single cycle. In fact, \cref{cor:SquareOfThreeBraidsAreOP} implies that order-preserving braids whose permutation type is a single cycle are abundant in $B_3$. These examples provide an affirmative answer to a question of Kin and Rolfsen \cite[Question 6.1]{KinRolfsen18} about bi-orderability of a class of link groups and yield a conjectural picture about the property of being order-preserving for 3-strand braids, see \cref{conj:ObstructionToOPBraidViaBurau}. 

\subsection{Organization of the article}

In \cref{sec:FiniteGroupOP}, we prove \cref{thm:FiniteOrderOP} and explain some of its corollaries. In \cref{sec:MotivatingExamples}, we explain an example to illustrate the ideas behind \cref{thm:OPViaBurauRep} and to motivate the discussion of some concepts in \cref{subsec:OrderedRingAndFields}, and \cref{subsec:puiseux}. In \cref{subsec:ProofMainTheorem}, we give a proof of the main theorem and apply it to study the order-preserving properties of braids in ]\cref{subsec:OP3Braid} and \cref{subsec:SporadicExamples}. We end with a conjecture and a question in \cref{subsec:ConjecturesQuestions}.

\subsection{Acknowledgement} The second author would like to thank Alan Reid for pointing out the result in \cite{Eckmann84} which extends \cref{thm:FiniteOrderOP} to the case of closed surface groups, George Domat for various conversations about finite subgroups of $\Out(F_n)$, Carl-Fredrik Nyberg Brodda for corrections on some references regarding Stallings' result. The authors would also like to thank the referee(s) for their careful readings. Their corrections, suggestions and feedback have improved the content and the exposition of our paper. The authors acknowledge the support of the thematic semester in Geometric Group Theory at CRM Montreal where the collaboration started.
The main idea of \cref{thm:OPViaBurauRep} was conceptualized at ICERM during the Braid Reunion workshop, which is supported by the National Science Foundation under Grant No. DMS-1929284.
Johnson acknowledges the support of the NSF grant DMS-2213213.
Le acknowledges the support of the AMS Simons Travel Grant and of Harry Baik and Korea Advanced Institute of Science and Technology (KAIST) during the preparation of the paper.  

\section{Classification of finite-order outer automorphisms of free and surface groups}
\label{sec:FiniteGroupOP}
In this section, we will give the proof of \cref{thm:FiniteOrderOP}. This proof is a combination of well-known results about virtually free groups in \cite{Stallings68} and virtually surface groups in \cite{Eckmann84}. 

\begin{proof}[Proof of \cref{thm:FiniteOrderOP}]
    Let $\Gamma := \pi^{-1}(H)$. We will first prove that if $\Gamma$ contains torsion then $H$ is not order-preserving. Let $f \in \Gamma$ be a non-trivial finite-order automorphism of $G$. Since $\pi(f)$ is a non-trivial element of $H$, it suffices to show that $\pi(f)$ is not order-preserving. Since $f$ is a non-trivial finite-order automorphism, there exists $w \in G$ such that the orbit of $w$ under $f$ is finite and contains at least 2 elements. It follows that there exists no linear order on the orbit of $w$ under $f$ that is $f$-invariant. Therefore, $\pi(f)$ and hence $H$ is not order-preserving. 
    
    Now we will show that if $\Gamma$ is torsion-free then $H$ is order-preserving. The key observation is that if $\Gamma$ is torsion-free then $\Gamma$ is isomorphic to either a free group or a bi-orderable surface group according to $G$ by the seminal work of Stallings \cite{Stallings68} and separately Eckmann, Linnell and Muller \cite{Eckmann84}. In particular, when $G$ is free, the group $\Gamma$ contains a free group as a finite index subgroup. If $\Gamma$ is also torsion-free, then the work of Stallings \cite{Stallings68} implies that $\Gamma$ must be a free group.
    
    When $G$ is a nonabelian bi-orderable closed surface group, the group $\Gamma$ contains a finite-index subgroup that is a nonabelian closed surface group. If $\Gamma$ is torsion-free, then $\Gamma$ is also a nonabelian surface group by the work of Eckmann, Linnell and Muller, see \cite[Corollary 1.2]{Eckmann84}. Since $G$ is neither $\mathbb{Z}/2\mathbb{Z}$ nor the Klein bottle group, $\Gamma$ is a biorderable closed surface group.
    
    In either case, we have the following short exact sequence
    \begin{equation}
    \label{eq:FiniteCyclicExtensionOfG}
        1 \to G \to \Gamma \to H \to 1
    \end{equation}
    where $\Gamma$ is a bi-orderable group. Let $f$ be any element of $\Gamma$, then we can view $f$ as an automorphism of $G$. The conjugation action of $f$ on $G$ in $\Gamma$ coincides with $f:G\to G$ as an automorphism of $G$. Consider $P$ the positive cone of any bi-order on $\Gamma$. We restricts $P$ to $G$ to obtain $P_G := P\cap G$ a bi-order on $G$. Since $P$ defines a bi-order on $\Gamma$, the set $P_G$ is invariant under conjugation by $f$. This implies that $f$ when viewed as an automorphism of $G$ preserves $P_G$. Therefore, $H$ is an order-preserving finite subgroup of $\Out(G)$. 
\end{proof}

\subsection{Properties of finite order-preserving subgroups}
As a consequence of the above theorem, we observe some necessary conditions for a finite subgroup of outer automorphisms of free and bi-orderable surface group to be order-preserving. 

\begin{cor}
    Let $G$ be a non-abelian finitely generated free group or a bi-orderable surface group and $\chi(G)$ be the Euler characteristic of a finite graph or of a closed surface with the fundamental group isomorphic to $G$. If $H \leq \Out(G)$ is an order-preserving finite subgroup, then $|H|$ divides $\chi(G)$. Furthermore, every divisor of $\chi(G)$ occurs as the order of an order-preserving finite subgroup of $\Out(G)$.
\end{cor} 

\begin{proof}
    We assume that $G$ is a finite-generated free group. The proof in the case of bi-orderable surface groups is identical. Recall that $\pi:\Aut(G) \to \Out(G)$ is the canonical surjective group homomorphism. We have the short exact sequence of groups
    \[
    1 \to G \to \Gamma \to H \to 1
    \]
    where $\Gamma := \pi^{-1}(H)$. The proof of \cref{thm:FiniteOrderOP} shows that if $H$ is order-preserving, then $\Gamma$ is a free group. 
    
    Let $X$ be a finite graph with the fundamental group isomorphic to $\Gamma$ and $Y$ be the cover corresponding to the subgroup $G \leq \Gamma$. The degree of the covering $Y \to X$ is $|H|$. Since $\chi(G)$ is the Euler characteristic of $Y$, we have $\chi(G)/|H|$ is the Euler characteristic of $X$ which is an integer since $X$ is a finite graph. Therefore, $|H|$ divides $\chi(G)$. 

    Now given a divisor $d$ of $\chi(G)$, we will construct an order-preserving finite subgroup of $\Out(G)$ of order $d$. Since $\chi(G)/d$ is an integer, there exists a finite graph $X$ with Euler characteristic $\chi(G)/d$. We consider a surjective group homomorphism $q:\pi_1(X) \to \mathbb{Z}/d\mathbb{Z}$. The kernel of this homomorphism is a subgroup of $\pi_1(X)$ isomorphic to $G$ since it has the same Euler characteristic as $G$. The action of $\mathbb{Z}/d\mathbb{Z}$ on $\ker(q) \cong G$ is conjugation by elements in $\pi_1(X)$. Since the centralizer of $\ker(q)$ in $\pi_1(X)$ is trivial, the action of $\mathbb{Z}/d\mathbb{Z}$ on $\ker(q) \cong G$ induces an embedding of $\mathbb{Z}/d\mathbb{Z}$ into $\Out(G)$. The proof of \cref{thm:FiniteOrderOP} shows that this finite subgroup is order-preserving.        
\end{proof}

\subsection{Connections with previous results}
We discuss how \cref{thm:FiniteOrderOP} is related to results in \cite{BoyerRolfsenWiest05} and generalizes results in \cite{KinRolfsen18}. 

Kin and Rolfsen \cite{KinRolfsen18} classified order-preserving braids whose Nielsen-Thurston type is periodic. See \cite[Theorem 4.10]{KinRolfsen18} and the previous paragraph for the statement of the theorem and the discussion of periodic braid. When the Nielsen-Thurston type of a braid is periodic, Kin and Rolfsen show that a braid is not order-preserving when the outer automorphism induced by the braid action on the free group, see \cref{eq:ArtinAction}, can be chosen to be an automorphism of the free group with finite order \cite[Equation 4.2]{KinRolfsen18}. This is precisely the same obstruction that we used in \cref{thm:FiniteOrderOP}.          

In \cite{BoyerRolfsenWiest05}, the authors classified among Seifert fibered spaces 3-manifolds with bi-orderable fundamental group, see \cite[Theorem 1.5]{BoyerRolfsenWiest05}. In \cref{thm:FiniteOrderOP}, if we assume that $H$ is a finite cyclic subgroup of $\Out(\pi_1(S))$ where $S$ is a closed orientable surface of genus $g \geq 2$, then \cref{thm:FiniteOrderOP} recovers a special case of \cite[Theorem 1.5]{BoyerRolfsenWiest05}. We explain the equivalence between our algebraic condition and the topological condition in \cite[Theorem 1.5]{BoyerRolfsenWiest05} in this case.

Recall that if $S$ is a closed orientable surface of genus $g \geq 2$, the extended mapping class group, $\MCG^{\pm}(S)$, is defined as the group of homeomorphisms of $S$ where two homeomorphisms are equivalent if they are isotopic. A mapping class element in $\MCG^{\pm}(S)$ induces a well-defined outer automorphism of $\pi_1(S)$. This correspondence gives a group homomorphism $\mathcal{I}:\MCG^{\pm}(S) \to \Out(\pi_1(S))$. By the Dehn-Nielsen-Baer Theorem, this group homomorphism is an isomorphism, see \cite[Chapter 8]{farbmargalit2012primer}
\begin{equation}
\label{eq:DehnNielsenBaerTheorem}
\mathcal{I}: \MCG^{\pm}(S) \to \Out(\pi_1(S)).
\end{equation}
For a finite cyclic subgroup $H \leq \Out(\pi_1(S))$, let $\varphi:S\to S$ be a homeomorphism representing a mapping class $[\varphi]$ such that $\mathcal{I}([\varphi])$ generates $H$. By the Nielsen realization theorem (see \cite{nielsen1943realization} and \cite[Theorem 5]{kerckhoff1983nielsenrealization}), we may assume that $\varphi:S\to S$ is a finite-order homeomorphism.  Also, see \cite[Chapter 7]{farbmargalit2012primer} for a historical discussion. The mapping torus $M_\varphi$ is defined as
\[
M_\varphi = \frac{S\times[0,1]}{(x,1) \sim (\varphi(x),0)}.
\]
The manifold $M_\varphi$ can be decomposed as a union of circles
\begin{equation}
\label{eq:SeifertDecomposition}
M_\varphi = \bigcup_{x \in S} \bigcup_{n=0}^{k} \{\varphi^n(x) \}\times[0,1]
\end{equation} 
where $k$ is the order of $\varphi$ in $\Homeo(S)$. This decomposition gives $M_\varphi$ a structure of a Seifert fibered space, see \cite[Section 3]{scott8geometries} for discussions on Seifert fibered space. The base space $B$ of the Seifert structure on $M_\varphi$ is obtained by identifying each circle in the decomposition in \cref{eq:SeifertDecomposition} to a point. In other words, the base space $B$ is the quotient of $S$ by the action of the finite-order homeomorphism $\varphi$. The space $B$ in general is a 2-orbifold. See \cite[Section 2]{scott8geometries} for a definition of an orbifold as well as covering space theory for orbifolds. The natural quotient map $S \to B$ is a regular orbifold covering map. The deck transformation group of $S\to B$ is generated by $\varphi$. Therefore, the (orbifold) fundamental group of $S$ and $B$ fit in the short exact sequence 
\begin{equation}
\label{eq:SESofBasespace}
1 \to \pi_1(S) \to \pi_1(B) \to \langle \varphi \rangle \to 1 
\end{equation}
where the action of $\varphi$ on $\pi_1(S)$ is the induced action $\mathcal{I}([\varphi]):\pi_1(S)\to\pi_1(S)$. Identifying $\langle \varphi \rangle$ and $H$ via $\mathcal{I}$ in \cref{eq:DehnNielsenBaerTheorem}, we can see that the fundamental group of the base space of the Seifert structure, $\pi_1(B)$, is isomorphic to $\Gamma = \pi^{-1}(H)$, the pre-image of $H$ in $\Aut(\pi_1(S))$. The condition that $\Gamma$ is torsion free in \cref{thm:FiniteOrderOP} happens precisely when the base space $B$ is an honest surface without any orbifold point. In other words, the Seifert fibered structure of $M_\varphi$ is locally trivial. Therefore, our condition is the algebraic analog of the condition that the total space of $M_\varphi$ is a locally trivial, orientable circle bundle over a surface with bi-orderable fundamental group \cite[Theorem 1.5]{BoyerRolfsenWiest05}.  

\subsection{A new construction of order-preserving outer automorphism}

As a final remark of this section, we indicate how to use \cref{thm:FiniteOrderOP} to produce new examples of infinite-order order-preserving outer automorphisms that do not necessarily satisfy the criterion of \cite[Theorem 2.6]{PerronRolfsen03}.

\begin{cor}\label{cor:OPOuterAutomorphismCommutingWithFiniteOrder}
    Let $G$ be a nonabelian free group or a bi-orderable surface group and let $H\leq\Out(G)$ be a finite order-preserving subgroup. Now let $\Gamma = \pi^{-1}(H)$ where $\pi:\Aut(G) \to \Out(G)$ is the canonical group homomorphism. Since the center of $G$ is trival, we identify $G$ with the subgroup of $\Aut(G)$ consisting of inner automorphisms. Suppose $f \in \Aut(\Gamma)$ is an order-preserving automorphism that also leaves $G$ invariant. Restricting $f$ to $G$, we obtain an induced outer automorphism $\bar{f} \in \Out(G)$. Then the subgroup $\langle H, \bar{f}\rangle$ is order-preserving. 
\end{cor}

\begin{proof}
    Since $H$ is order-preserving, \cref{thm:FiniteOrderOP} shows that $\Gamma = \pi^{-1}(H)$ is torsion-free and is a free group or a bi-orderable surface group containing $G$. The proof of \cref{thm:FiniteOrderOP} shows that $H$ preserves every bi-order on $G$ obtained by restricting any bi-order of $\Gamma$. By hypothesis, there exists a bi-order $P$ on $\Gamma$ that $f$ preserves. Therefore, the subgroup $\langle H, \bar{f}\rangle$ preserves the bi-order $P\cap G$ on $G$. 
\end{proof}

Here is a concrete example illustrating \cref{cor:OPOuterAutomorphismCommutingWithFiniteOrder}.

\begin{example}
    We consider the free group $\Gamma = \langle a,b \rangle$  and $f\in \Aut(\Gamma)$ defined by 
    \[f(a) = aba \quad f(b) =ba.\] 
    This is well-known to be the monodromy of the fundamental group of the figure-8 knot complement. The induced action of $f$ on $H_1(\Gamma;\mathbb{Z})$ is a matrix with two positive eigenvalues $\frac{3\pm \sqrt{5}}{2}$. By \cite[Theorem 2.6]{PerronRolfsen03}, the automorphism $f$ is order-preserving. Now consider a group homomorphism $\varphi:\Gamma \to \mathbb{Z}/3\mathbb{Z}$ given by $\varphi(a) = \varphi(b) = \bar{1}$ and let $G= \ker(\varphi)$. The subgroup $G$ has index 3 and is isomorphic to the free group of rank $4$. We can choose a free basis of $G$ by looking at the corresponding covering space of the rose graph with two petals. See \cref{fig:CoverWithFundamentalGroupG}.

    We choose a basepoint to be the topmost vertex in the graph. The thickened blue edges on the right and at the bottom denote a maximal tree of the graph. A free basis of $G$ is $\{w = bA, x = bbAB, y= bba,z=bbb\}$ where capital letters denote inverses. Let $\rho:G\to G$ be the action of $a$ on $G$ by conjugation which is given by
    
    \begin{align*}
    \rho(w) &= awA = Wxw, \quad\rho(x) = axA = WzYw\\
    \rho(y) &= ayA = Wz,
    \quad\rho(z) = azA = Wzw
    \end{align*}
    
    \begin{center}
    \begin{figure}[h!]
        \begin{tikzpicture}[decoration={markings, mark= at position 0.5 with {\arrow{stealth}}}]
            \draw[very thick, blue, postaction={decorate} ] (0,1) -- ({sqrt(3)/2},-1/2);
            \draw[very thick, blue, postaction={decorate}] ({sqrt(3)/2},-1/2) -- ({-sqrt(3)/2},-1/2);
            \draw[thick, blue, postaction={decorate}] ({-sqrt(3)/2},-1/2) -- (0,1);
            
            \draw[thick, red, postaction={decorate}] (0,1) arc[start angle=90, end angle=-30, radius=1];
            \draw[thick, red, postaction={decorate}] ({sqrt(3)/2},-1/2) arc[start angle=-30, end angle=-150, radius=1];
            \draw[thick, red, postaction={decorate}] ({-sqrt(3)/2},-1/2) arc[start angle=-150, end angle=-270, radius=1];
            
            \filldraw[black] (0,1) circle (2pt);
            \filldraw[black] ({sqrt(3)/2},-1/2) circle (2pt);
            \filldraw[black] (-{sqrt(3)/2},-1/2) circle (2pt);
            \node[] at (0, -0.7) {$b$};
            \node[] at ({-sqrt(3)/2*0.7},1/2*0.7) {$b$};
            \node[] at ({sqrt(3)/2*0.7},1/2*0.7) {$b$};
            \node[] at (0, -1.2) {$a$};
            \node[] at ({-sqrt(3)/2*1.2},1/2*1.2) {$a$};
            \node[] at ({sqrt(3)/2*1.2},1/2*1.2) {$a$};
        \end{tikzpicture}
        \caption{A finite graph with fundamental group isomorphic to $G$. The basepoint is chosen to be the topmost vertex of the graph. A maximal tree contains two thickened blue edges. The edges are labeled by its image under the covering map to the rose graph with two edges $a$ and $b$.}
        \label{fig:CoverWithFundamentalGroupG}
    \end{figure}    
    \end{center}
    \noindent

    \noindent
    The automorphism $f$ does not preserve $G$, but $f^2$ does. A direct computation shows that
    \begin{align*}
    f^2(w) &= Yw, \quad f^2(x) = Yx \\ 
    f^2(y) &= XyXzWXzWXzWyWy, \quad f^2(z) = XyXzWXzWyWy 
    \end{align*}
    As a consequence of \cref{cor:OPOuterAutomorphismCommutingWithFiniteOrder}, the automorphism $h = f^2\circ \rho$ is order-preserving. Using the image of the elements $\{w,x,y,z\}$ as a basis for $H_1(G;
    \mathbb{Z})$, $h_{\ab}:H_1(G;
    \mathbb{Z}) \to H_1(G;
    \mathbb{Z})$ is represented by the matrix
    \[
    \begin{pmatrix}
        0 & 1 & -4 & -3 \\
        1 & 1 & -3 & -3 \\
        -1 & 0 & 4 & 3 \\
        0 & -1 & 2 & 2 \\
    \end{pmatrix}
    \]
    which has $-1$ as an eigenvalue. Therefore, $h$ gives an example of an automorphism of a free group on 4 letters that induces a non-trivial outer automorphism, is order-preserving and does not satisfy the hypothesis of \cite[Theorem 2.6]{PerronRolfsen03}.
\end{example}

\section{Motivating examples}
\label{sec:MotivatingExamples}
We start with a simple example that illustrates the main ideas behind our strategy. The automorphism in the following example is known to be order-preserving by the work of Perron and Rolfsen in \cite{PerronRolfsen03}. We will give a different proof that it is order-preserving to motivate our strategy. 

The example comes from the Artin action of braid groups $B_n$ on $F$, the free group on $n$ generators. Recall the braid group $B_n$ is generated by $n-1$ generators $\sigma_i$ for $1\leq i \leq n-1$ subject to the following relations:
\begin{equation}
\label{eq:BraidRelations}
\begin{aligned}
    \sigma_i \sigma_j &= \sigma_j \sigma_i && \text{if } |i-j|>2, \\
    \sigma_i \sigma_{i+1} \sigma_i &= \sigma_{i+1}\sigma_i\sigma_{i+1} && \text{for } 1 \leq i \leq n-2. \\
\end{aligned}    
\end{equation}
The Artin action of the braid group $B_n$ on the free group $F$ of rank $n$, $\Theta:B_n \to \Aut(F)$, is given by the following formula
\begin{equation}
\label{eq:ArtinAction}
    \Theta(\sigma_i)(x_j)= 
    \begin{cases}
        x_i \ x_{i+1} \ x_i^{-1} & \text{if $j=i$} \\
        x_i & \text{if $j=i+1$} \\
        x_j & \text{if $|i-j| \geq 2$ } 
    \end{cases}
\end{equation}
where $\{x_j \mid 1 \leq j \leq n\}$ is a free basis of $F$.

\subsection{The braid $\sigma_1^2$ in $B_3$ is order-preserving.}
In this example, let $F$ be the free group of rank 3 and consider the automorphism $\Theta(\sigma_1^2) \in \Aut(F)$ where $\sigma_1 \in B_3$. For convenience, we let $x = x_1$, $y = x_2$, $z = x_3$ and $\varphi = \Theta(\sigma_1^2)$. In this notation, $\varphi$ is given by
\[
\varphi(x) = xy x y^{-1}x^{-1} \quad \varphi(y) = xyx^{-1} \quad \text{and} \quad \varphi(z)= z . 
\]
Since the induced action of $\varphi$ on $H_1(F;\mathbb{Z})$ is the identity, it follows that $\varphi$ is order-preserving by \cite[Theorem 2.6]{PerronRolfsen03}. We outline an alternative proof here.

We consider the short exact sequence of groups
\begin{equation}
    1 \to K \to F \xrightarrow{\mu} \mathbb{Z} \to 0
\end{equation}
given by $\mu(x) = \mu(y) = \mu(z)= 1$. Let $\tau \in \Aut(K)$ be an automorphism induced by the conjugation action of $x$ on $K$.

\begin{fact}
\label{fact:ConstructingBiOrdersOnFGFreeGroup}
    If $P \subset K$ defines a bi-order on $K$ and is $\varphi$ and $\tau$-invariant, then
    \begin{equation}
        \label{eq:OrderOnF2}
        P \cup \mu^{-1}(\{k \mid k >0\})
    \end{equation}
defines a bi-order on $F$ that is $\varphi$-invariant. 
\end{fact}

\begin{proof}
If $P$ defines a bi-order on $K$ that is preserved by $\tau$, then $P \cup \mu^{-1}(\{k \mid k >0\})$
defines a bi-order on $F$. Since $\mu\circ\varphi = \mu$, the kernel $K$ of $\mu$ is preserved by $\varphi$. Furthermore, if $P$ is also preserved by $\varphi|_K$, then the bi-order defined in \cref{eq:OrderOnF2} on $F$ is preserved by $\varphi$.
\end{proof}
Therefore, to show that $\varphi$ is order-preserving, we will construct a bi-order $P$ on $K$ that is invariant under $\varphi$ and $\tau$.
Using the lower central series of $K$, we can define an abundance of bi-orders on $K$. Let $K_1 = K$ and $K_m =[K_{m-1},K]$ be the $m^{th}$ term in the lower central series of $K$.

\begin{fact}
\label{fact:ConstructingStandardOrderOnFreeGroup}
    To specify a bi-order $P$ on $K$, it suffices to specify a left-order $P_m$ on each quotient $K_m/K_{m+1}$. 
\end{fact}

\begin{proof}
The key properties of the lower central series of $K$, a free group (finitely generated or not) are:
\begin{itemize}
    \item The series is cofinal $\displaystyle\bigcap_{m=1}^\infty K_m = \{1\}$.
    \item The successive quotient $K_m/K_{m+1}$ is abelian for all $m \geq 1$.
    \item The induced action of $K$ on $K_m/K_{m+1}$ by conjugation is trivial. 
\end{itemize}
Let $P_m$ be any left-order on $K_m/K_{m+1}$, and hence bi-order since $K_m/K_{m+1}$ is abelian. A bi-order $P$ on $K$ is constructed from the lower central series as follows. Since the lower central series of $K$ is cofinal, for any $\gamma \in K$ there exists a smallest $m$ such that the image of $\gamma$ in $K_{m}/K_{m+1}$ is non-trivial. We say that $\gamma$ is in $P$ if and only if the image of $\gamma$ in $K_{m}/K_{m+1}$ is in $P_m$. Since the lower central series is cofinal, every non-trivial element of $K$ is either in $P$ or not in $P$. The set $P$ forms a semigroup in $K$ since $P_m$ are all semigroups in the corresponding quotients. The fact that the induced action of $K$ on $K_m/K_{m+1}$ by conjugation is trivial implies that $P$ defines a bi-order on $K$.
\end{proof}

Since $K_m$ is a characteristic subgroup of $K$, the subgroup $K_m$ is preserved by any automorphism $\psi$ of $K$. The automorphism $\psi$ induces an automorphism $\psi_m$ on the quotient $K_m/K_{m+1}$. It follows directly from the construction above that if $P_m$  is invariant under $\psi_m$ for all $m \geq 1$, then the bi-order $P$ is also $\psi$-invariant. Therefore, to find a bi-order $P$ on $K$ that is $\varphi$- and $\tau$- invariant, we want to find a family of left-orders $\{P_m\}_{m \geq 1}$ such that each $P_m$ is $\varphi_m$ and $\tau_m$ -invariant for all $m$.

\begin{fact}
\label{fact:LowerCentralQuotientsEmbedInTensorPowers} For any positive integer $m$, the lower central series quotient $K_m/K_{m+1}$ embeds in $H_1(K;\mathbb{Z})^{\otimes_{\mathbb{Z}} m}$. We refer to the embedding as $\iota_m$.  Furthermore for any $\psi\in \Aut(K)$, the induced automorphism $\psi_m$ on $K_m/K_{m+1}$ agrees with the restriction of the automorphism $\psi_{\ab}^{\otimes_\mathbb{Z}m}$ on $H_1(K;\mathbb{Z})^{\otimes_{\mathbb{Z}} m}$ to the image of $\iota_m$.     
\end{fact}
\noindent
For a proof of \cref{fact:LowerCentralQuotientsEmbedInTensorPowers} see \cite[Lemma 3.1]{CaiClayRolfsen} and \cite{PerronRolfsen03}. The appropriate generalization of this fact is \cref{lem:EmbeddingTheLowerCentralQuotients}.

We describe explicitly the action of $\tau_{\ab}^{\otimes_\mathbb{Z}m}$ and $\varphi_{\ab}^{\otimes_\mathbb{Z}m}$ on the tensor power $H_1(K;\mathbb{Z})^{\otimes_{\mathbb{Z}} m}$. The automorphism $\tau \in \Aut(K)$ is given by the conjugation action of $x$ on $K$. More explicitly, we have 
\[\tau(w) = xwx^{-1}\] 
for all $w \in K.$ Let $u_i = [x^{i+1}y^{-1}x^{-i}]$ and $v_i = [x^{i}yz^{-1}x^{-i}]$ be the image of the corresponding elements in $H_1(K;\mathbb{Z})$. The homology group $H_1(K;\mathbb{Z})$ is an infinitely-generated free abelian group with a basis $\{u_i,v_i\mid i \in \mathbb{Z}\}$. If we write $t^i \ w =\tau^i_{\ab}(w)$ for any $w \in H_1(K;\mathbb{Z})$ and $i \in \mathbb{Z}$, then $u_i = t^i \  u_0 $ and $v_i = t^i \ v_0$ for any integers $ i\in \mathbb{Z}$. For an arbitrary element of $H_1(K;\mathbb{Z})$, we have
\begin{equation}
    \sum_{i\in\mathbb{Z}} (c_i u_i + d_i v_i)  = \left(\sum_{i\in\mathbb{Z}}c_i t^i\right) \ u + \left(\sum_{i\in\mathbb{Z}}d_i t^i\right) \ v
\end{equation}
where the coefficients $c_i$ and $d_i$ are integers such that all but finitely many are zero. The homology group $H_1(K;\mathbb{Z})$ becomes a finitely generated free $\mathbb{Z}[t^{\pm 1}]$-module with a free basis $\{u  = u_0,v = v_0\}$. The action of $\tau_{\ab}$ on $H_1(K;\mathbb{Z})$ is given by  
\begin{equation}
    \tau_{\ab}(u_i) = u_{i+1} \quad \text{and} \quad \tau_{\ab}(v_i) = v_{i+1}
\end{equation}
on the basis and is extended $\mathbb{Z}$-linearly on $H_1(K;\mathbb{Z})$. Therefore, the action of $\tau_{\ab}$ on the $\mathbb{Z}[t^{\pm 1}]$-module $H_1(K;\mathbb{Z})$ is multiplication by $t$. Since $\mu =\mu\circ \varphi$, we can write $\varphi(x) = xk$ for some $k \in K$. For the image $[w] \in H_1(K;\mathbb{Z})$ of any $w \in K$, we have
\begin{equation}
\label{eq:CheckingVarphiAndTauCommute}
\varphi_{\ab}\tau_{\ab}([w]) = \varphi_{\ab}([xwx^{-1}]) = [xk\varphi(w)k^{-1}x^{-1}] =  \tau_{\ab}([\varphi(w)])=\tau_{\ab}\varphi_{\ab}([w])  
\end{equation}
That is, the action of $\varphi_{\ab}$ and $\tau_{\ab}$ commute. We can view $\varphi_{\ab}$ as a $\mathbb{Z}[t^{\pm 1}]$-module automorphism. A direct computation shows that
\[ 
\varphi_{\ab}(u) = t^2 \ u \text{ and }\varphi_{\ab}(v) = (-t+1) \ u + v .
\]
Writing $u$ and $v$ as basis column vectors of the $\mathbb{Z}[t^{\pm 1}]$-module $H_1(K;\mathbb{Z})$, we can represent $\varphi_{\ab}$ by the matrix
\[
\begin{pmatrix}
    t^2 & -t + 1 \\ 0 & 1 
\end{pmatrix}.
\]

The tensor power of $H_1(K;\mathbb{Z})$ has the following natural module structure.
\begin{fact}
\label{fact:BasisOfTensorPowerOfHomology}
The tensor power $H_1(K;\mathbb{Z})^{\otimes_{\mathbb{Z}}m}$ has a structure of a finitely-generated free $\mathbb{Z}[t^{\pm 1}]^{\otimes_\mathbb{Z} m}$-module. On simple tensors, the ring action is defined for all $r_i \in \mathbb{Z}[t^{\pm 1}]$ and $w_i \in H_1(K;\mathbb{Z})$ by the formula
\[
(r_1 \otimes \dots \otimes r_m)(w_1 \otimes \dots \otimes w_m) = (r_1 w_1 \otimes \dots \otimes r_m w_m)
\]
and is extended linearly to the whole ring. A basis for this module can be chosen to be
\begin{equation}
    \mathcal{B}_m = \{w_{1}\otimes w_{2} \otimes \dots \otimes w_{m} \mid w_{i} \in \{u,v\}\}.
\end{equation}    
\end{fact}
\noindent In the general setting, we describe the module structure after \cref{prop:HomologyModuleBurau} and proved the claim about the basis of the module in \cref{prop:TensorPowerOfHomologyModuleIsFree}. The induced action of $\tau_{\ab}$ on $H_1(K;\mathbb{Z})^{\otimes_{\mathbb{Z}}m}$ is multiplication by $t^{\otimes_{\mathbb{Z}}m}$. We order the elements in $\mathcal{B}_m$ by the following rule
\begin{equation}
\label{eq:OrderOnTheBasis}
w_1 \otimes \dots \otimes w_m \prec w'_1 \otimes \dots \otimes w'_m
\end{equation}
if and only if $w_i = u$ and $w_i' = v$ and $w_j = w_j'$ for all $1 \leq j < i$. The induced action of $\varphi_{\ab}$ can be represented by an upper-triangular matrix with respect to the ordered basis $\mathcal{B}_m$. For instance, when $m = 2$ we have an ordered basis $u \otimes u \prec u \otimes v \prec v \otimes u \prec v \otimes v$, and the matrix representing 
$\varphi_{\ab}^{\otimes_{\mathbb{Z}} 2}$ is given by
\[
\begin{pmatrix}
    t^2 \otimes t^2 & t^2 \otimes (-t+1) & (-t+1) \otimes t^2 & (-t+1) \otimes (-t+1) \\ 
    0 & t^2 \otimes 1 & 0 & (-t+1) \otimes 1 \\
    0 & 0 & 1 \otimes t^2 & 1 \otimes (-t+1) \\
    0 & 0 & 0 & 1 \otimes 1 \\
\end{pmatrix}. 
\]
It is clear from this matrix representation that the eigenvalues of $\varphi_{\ab}^{\otimes_{\mathbb{Z}} m}$ are precisely
\[
\{\lambda_1 \otimes \dots \otimes \lambda_m \mid \lambda_i \in \{1,t^2\}\}.
\]

The final fact that we need to describe the order $P_m$ on $H_1(K;\mathbb{Z})^{\otimes_\mathbb{Z} m}$ is the following.  
\begin{fact}
\label{fact:OrderingOnTensorProductOfLaurentRing}
    The ring $\mathbb{Z}[t^{\pm 1}]^{\otimes_\mathbb{Z} m}$ admits an ordering defined by a subset $Q_m$ such that $(\mathbb{Z}[t^{\pm 1}]^{\otimes_\mathbb{Z} m},Q_m)$ is an ordered (additive) abelian group and that $Q_m \cdot Q_m \subset Q_m$ and $q_1 \otimes \dots \otimes q_m \in Q_m$ for any $q_1,\dots,q_m \in Q_1$ . 
\end{fact}
\begin{proof}[Sketch of the proof]
Elements of $\mathbb{Z}[t^{\pm 1}]^{\otimes_\mathbb{Z} m}$ are finite $\mathbb{Z}$-linear combinations of elements in the set
\[
\mathcal{S}_m = \{t^{n_1} \otimes \dots \otimes t^{n_m} \mid n_i \in \mathbb{Z}\}.
\]
The set $\mathcal{S}_m$ as a multiplicative group is isomorphic to $\mathbb{Z}^m$. We order $\mathcal{S}_m$ lexicographically by declaring that the positive elements are
\[
\{t^{n_1} \otimes \dots \otimes t^{n_m} \mid n_j > 0 \text{ and } n_i = 0 \ \text{ for } \ 1 \leq i \leq j-1\}
\]
The set of positive elements $Q_m \subset \mathbb{Z}[t^{\pm 1}]^{\otimes_\mathbb{Z} m}$ is defined to be the set of elements such as the coefficient of the smallest term with respect to the ordering on $\mathcal{S}_m$ is positive in $\mathbb{Z}.$ It is straightforward to check that $Q_m$ defines an ordering on $\mathbb{Z}[t^{\pm 1}]^{\otimes_\mathbb{Z} m}$ as an abelian group and satisfies $Q_m \cdot Q_m \subset Q_m$ and $q_1 \otimes \dots \otimes q_m \in Q_m$ for any $q_1,\dots,q_m \in Q_1$ . 
\end{proof}

\noindent
The appropriate generalization of \cref{fact:OrderingOnTensorProductOfLaurentRing} is stated and proven in \cref{lem:OrderingOnE} and \cref{prop:OrderingOnTensorPowerOfPuiseuxField}. 

The ordering $P_m \subset H_1(K;\mathbb{Z})^{\otimes_\mathbb{Z} m}$ can be defined as follows. Since $H_1(K;\mathbb{Z})^{\otimes_\mathbb{Z} m}$ is a free $\mathbb{Z}[t^{\pm 1}]^{\otimes_\mathbb{Z} m}$-module by \cref{fact:BasisOfTensorPowerOfHomology}, we write elements of $H_1(K;\mathbb{Z})^{\otimes_\mathbb{Z} m}$ as column vectors $\left(r_1, \dots, r_{2^m} \right)^\top$ where $r_i \in \mathbb{Z}[t^{\pm 1}]^{\otimes_\mathbb{Z} m}$. In particular, we choose the basis $\mathcal{B}_m$ for $H_1(K;\mathbb{Z})^{\otimes_\mathbb{Z} m}$ with an ordering defined by \cref{eq:OrderOnTheBasis}. The set of positive elements are
\[
P_m = \{\left(r_1, \dots, r_{2^m} \right)^\top \mid r_j \in Q_m \text{ and } r_i = 0 \ \text{ for } \ j< i\leq 2^m \}.
\]
In other words, a column vector in $H_1(K;\mathbb{Z})^{\otimes_\mathbb{Z} m}$ is in $P_m$ if the bottom-most non-zero entry is in $Q_m$. Since $t^{\otimes_\mathbb{Z} m} \in Q_m$ and  $Q_m \cdot Q_m \subset Q_m$, we have $\tau_{\ab}^{\otimes_\mathbb{Z} m}(P_m) = t^{\otimes_\mathbb{Z} m} \cdot P_m \subset P_m$. That is, $P_m$ is $\tau_{\ab}^{\otimes_\mathbb{Z} m}$-invariant for all $m$. As remarked after \cref{fact:BasisOfTensorPowerOfHomology}, the induced action of $\varphi_{\ab}$ on $H_1(K;\mathbb{Z})^{\otimes_\mathbb{Z} m}$ with respect to the ordered basis $\mathcal{B}_m$ is left multiplication by an upper-triangular matrix where the diagonal entries are in $Q_m$. Therefore, the induced action of $\varphi_{\ab}$ on $H_1(K;\mathbb{Z})^{\otimes_\mathbb{Z} m}$ preserves $P_m$ for all $m$. The ordering $P_m$ is invariant under the induced action of both $\varphi_{\ab}$ and $\tau_{\ab}$. The bi-order $P$ on $K$ defined by $P_m$ as in \cref{fact:ConstructingStandardOrderOnFreeGroup} is invariant under both $\varphi$ and $\tau$. \cref{fact:ConstructingBiOrdersOnFGFreeGroup} gives a bi-order on $F$ that is $\varphi$-invariant. 

\subsection{Remarks about the example and the generalization}
The example of the action by $\varphi = \Theta(\sigma_1^2)$ suggests that if the induced action of a braid of $H_1(K;\mathbb{Z})$ has all positive eigenvalues, then one can define a family of bi-orders $P_m$ lexicographically using the eigenbases for the action. 

The main issue is eigenvalues of a matrix over $\mathbb{Z}[t^{\pm 1}]$, in general, lie in a finite extension of the field of fraction of $\mathbb{Z}[t^{\pm 1}]$. One cannot expect that these field extensions admit any ordering. For instance, the field $\mathbb{Q}(i)$ admits no ordering since it contains elements whose square are negative. To appropriately formulate the condition of positive eigenvalues, we use the theory of real fields due to Artin and Schreier. We discuss some necessary elements of this theory in \cref{subsec:OrderedRingAndFields}. In particular, we extend the field of scalars to a real closed field which is a maximal algebraic extension that admits a (unique) ordering \cref{thm:ExistenceOfRealClosureAndProperties}. We introduce this field in \cref{subsec:puiseux}. We verify that the field still gives the homology group $H_1(K;\mathbb{Z})$ the module structure with necessary properties in \cref{lem:OrderingOnE}, \cref{prop:OrderingOnTensorPowerOfPuiseuxField}, \cref{prop:TensorPowerOfHomologyModuleIsFree} and \cref{lem:EmbeddingTheLowerCentralQuotients} to carry out the argument outlined in this example.

\section{Preliminaries}

\subsection{Ordered rings and fields}
\label{subsec:OrderedRingAndFields}
In this section, we summarize the relevant results about ordered rings and fields following \cite[Chapter XI]{LangAlgebra}. The theory developed in this section is due to Artin and Schreier.

\begin{defi}
    Let $\mathbb{K}$ be a field (resp. ring). An \textbf{ordering} on $\mathbb{K}$ is a subset $P$ of $\mathbb{K}$ such that
    \begin{itemize}
        \item[(ORD 1)]\label{item:ORD1} $\mathbb{K} = P\sqcup \{0\} \sqcup -P$
        \item[(ORD 2)]\label{item:ORD2} If $x,y\in P$, then $x+y$ and $xy$ are both in $P$.
    \end{itemize}
The pair $(\mathbb{K},P)$ is called an \textbf{ordered field} (resp. \textbf{ordered ring}). The set $P$ is referred to as a \textbf{positive cone of } $\mathbb{K}$. 
\end{defi}

Let $(\mathbb{K},P)$ be an ordered field. Then the positive cone $P$ determines a total order $<_P$ on $\mathbb{K}$ by declaring that: 
\[
x<_Py \text{ if and only } y-x \in P.
\]
Conversely, any total order $<$ on a field $\mathbb{K}$ that satisfies 
\begin{itemize}
    \item[(ORD 3)] for all $x,y,z \in \mathbb{K},$ if $x < y$, then $x + z < y+z$ 
    \item[(ORD 4)] for all $x,y,z\in \mathbb{K},$ if $x < y$ and $0 < z $, then $zx < zy$ 
\end{itemize}
defines a positive cone of an ordering on $\mathbb{K}$ by considering 
\[P_< = \{x\in \mathbb{K} \mid x>0\}.\]
Furthermore, the correspondences 
\[
P \longmapsto <_P \text{ and } < \longmapsto P_<
\] 
are inverses of each other and define a bijective correspondence between the set of orderings on $\mathbb{K}$ and the set of total orders on $\mathbb{K}$ satisfying (ORD 3) and (ORD 4). The order $<_P$ satisfies properties similar to the properties of the natural ordering of the real numbers $\mathbb{R}$. 

\begin{lem}\cite[\S 1]{LangAlgebra}
    \label{lem:PropertiesOrderedField}
    Let $(\mathbb{K},P)$ be an ordered field. Then
    \begin{itemize}
        \item $\mathbb{K}$ has characteristic zero. 
        \item $x <_P y$ and $y <_P z$  implies $x <_P z$
        \item $x <_P y$ and $0 <_P z$  implies $xz <_P yz$
        \item $x <_P y$ and $0 <_P x,y$  implies $1/y <_P 1/x$
    \end{itemize}
\end{lem}

\begin{proof}
    Since $1 = (-1)^2$, we must have $1 \in P$. This implies that $n = 1 + \cdots + 1 \in P$ and, therefore, is zero if and only if $n=0$. Therefore, $\mathbb{K}$ has characteristic zero. Suppose that $x<_P y$. If $y <_P z$, then $z-x = z-y + y-x \in P$ which implies that $x<_Pz$. If we have $0<_Pz$, then $zy-zx =z(y-x) \in P$ which implies that $xz <_P yz$. If we have $0<_Px,y$, then $1/x - 1/y = (y-x)/xy \in P$ which implies that $1/y <_P 1/x$.
\end{proof}

\begin{lem}
    \label{lem:OrderedRingIsAnID}
    Let $(R,P)$ be a commutative ordered ring with unit. Then $R$ is an integral domain. Furthermore, there exists a unique ordering on the field of fractions of $R$ extending the ordering on $R$. 
\end{lem}

\begin{proof}
    Since $P^2 \subseteq P$, $P(-P) \subseteq -P$ and $(-P)^2 \subseteq P$, the ring $R$ has no zero divisors. It follows that $R$ embeds in its field of fractions $\mathbb{K} = \Frac(R)$. We define the positive cone on $\mathbb{K}$ to be
    \[
    Q = \left\{ \frac{a}{b} \mid ab \in P \right\}.
    \]
    We first check that the set $Q$ is well-defined. Let $a/b\in Q$. Suppose that $a/b = c/d $ for $c,d \in R$ and $b,d \neq 0$. It follows that $ad = bc$ which gives $cd a^2 = ab c^2$. We must have $c\neq 0$. Otherwise, the equalities $ad = bc = 0$ implies that $a =0$, contradicting the statement that $ab \in P$. Now $a^2,c^2,ab \in P$, together with $cd a^2 = ab c^2$ implies that $cd \in P$. Thus, $c/d \in Q$ so the set $Q$ is well-defined.
    
    We now check that (ORD 1) holds for $Q$. The negation of $Q$ is
    \[
        -Q = \left\{-\frac{a}{b} \mid ab \in P \right\} = \left\{\frac{c}{d} \mid cd \in -P \right\}
    \]
    which implies that $Q$ and $-Q$ are disjoint and that $\mathbb{K} = Q \sqcup \{0\} \sqcup -Q$. 
    
    Finally, we verify (ORD 2) for $Q$. Now, let $a/b$ and $c/d$ be in $Q$, we have $abcd \in P$ and $abd^2 +cdb^2 \in P$ which imply that 
    \[
    \frac{ac}{bd} \in Q \text{ and } \frac{a}{b} + \frac{c}{d} = \frac{ad + bc}{bd} \in Q.
    \]
    \cref{lem:PropertiesOrderedField} implies that $Q$ is the only possible positive cone that extends $P$ since $a/b = ab/b^2$. 
\end{proof}

\begin{defi}\label{defi:RealField}
 A field $\mathbb{K}$ is said to be \textbf{real} if $-1$ is not a sum of squares in $\mathbb{K}$. A field $\mathbb{K}$ is said to be \textbf{real closed} if any algebraic extension of $\mathbb{K}$ which is real must be equal to $\mathbb{K}$.  
\end{defi}

The following proposition gives a criterion for a field to be real closed.

\begin{prop}
    \label{prop:CriterionToBeRealClosed}
    {\cite[\S 2, Proposition 2.4]{LangAlgebra}}
    Let $\mathbb{K}$ be a field that is not algebraically closed, but the algebraic closure of $\mathbb{K}$ is $\mathbb{K}(\sqrt{-1})$. Then $\mathbb{K}$ is real and hence is real closed.
\end{prop}

Given a field $\mathbb{K}$, a \emph{real closure} of $\mathbb{K}$ is a real closed field that is algebraic over $\mathbb{K}$. The following theorem asserts that a real closure of a real field $\mathbb{K}$ always exists and is, in fact, an ordered field. 

\begin{thm}
    \label{thm:ExistenceOfRealClosureAndProperties}
    {\cite[\S 2, Theorem 2.2]{LangAlgebra}}
    Let $\mathbb{K}$ be a real field. Then there exists a real closure of $\mathbb{K}$. If $\mathbb{L}$ is real closed, then $\mathbb{L}$ has a unique ordering. The positive elements are precisely the squares of $\mathbb{L}$. Every polynomial of odd degree in $\mathbb{L}[X] $ has a root in $\mathbb{L}$. The algebraic closure of $\mathbb{L}$ is $\mathbb{L}(\sqrt{-1})$.
\end{thm}

It follows from the definitions that an ordered field is an example of a real field. As a consequence of \cref{thm:ExistenceOfRealClosureAndProperties}, a real field $\mathbb{K}$ admits an ordering, making it an ordered field.  This ordering is obtained by restricting the unique ordering from a real closure of $\mathbb{K}$. 

We conclude with the following criterion for a linear automorphism of a finite-dimensional vector space $V$ over an ordered field to preserve a bi-order of the underlying abelian group of the vector space $V$.

\begin{prop}
    \label{prop:PositiveEigenvaluesImpliesOP}
    Let $\mathbb{K}$ be an ordered field, and let $V$ be a finite-dimensional vector space over $\mathbb{K}$. Suppose that $\varphi: V \to V$ is a $\mathbb{K}$-linear automorphism such that all its eigenvalues $\lambda_i$ are positive in $\mathbb{K}$. Then $\varphi$ is an order-preserving $\mathbb{K}$-linear automorphism. In other words, $\varphi$ preserves an ordering of the underlying abelian group of the vector space $V$.    
\end{prop}

\begin{proof}
    Let $d = \dim_\mathbb{K}(V)$. Since all the eigenvalues of $V$ are in $\mathbb{K}$, there exists a basis $\{v_1,\dots,v_d\}$ for $V$ such that $\varphi$ can be represented by an upper triangular matrix with its eigenvalues on the diagonal. For all $1\leq i \leq d$, we have
    \[
    \varphi(v_i) = \lambda_i v_i + \sum_{j=1}^{i-1}c_{ij}v_j 
    \]
    where $c_{ij} \in \mathbb{K}$. We consider the subset $P \subset V$ defined by
    \[
    P = \left\{ \sum\limits_{i=1}^k a_i v_i \mid a_k > 0 \text{ and }1 \leq k \leq d\right\}.
    \]
    It follows that $V = P \sqcup \{0\} \sqcup -P $ and that $P+P\subset P$. Thus, $P$ defines a positive cone of the underlying abelian group of $V$. To see that $\varphi(P) \subset P$, we compute
    \begin{align*}
        \varphi\left(\sum\limits_{i=1}^k a_i v_i \right) 
    &= \sum\limits_{i=1}^k a_i \varphi(v_i) \\
    &= \sum\limits_{i=1}^k a_i (\lambda_i v_i + \sum_{j=1}^{i-1}c_{ij}v_j ) \\
    &= \left(\sum\limits_{i=1}^{k-1} a_i (\lambda_i v_i + \sum_{j=1}^{i-1}c_{ij}v_j )\right) + \left(\sum_{j=1}^{k-1}a_kc_{kj}v_j\right) + a_k\lambda_k v_k \in P.
    \end{align*}
    Thus, $\varphi$ preserves the positive cone $P$ of $V$ and $\varphi$ is order-preserving. 
\end{proof}

\subsection{The field of Puiseux series and its ordering} \label{subsec:puiseux}

Let $\mathbb{K}$ be a field. A formal power series over $\mathbb{K}$ is a series of the form
\[
\sum_{i = m}^\infty c_i t^{i}
\]
where $m\in \mathbb{Z}$, $c_i \in \mathbb{K}$ and $c_m \neq 0$. The field of formal power series over $\mathbb{K}$ is denoted as $\mathbb{K}((t))$. A Puiseux series with coefficient in $\mathbb{K}$ is a formal power series of the form
\[
f(t) = \sum_{i = m}^\infty c_i t^{i/n}
\]
where $n\in \mathbb{N}$, $m\in \mathbb{Z}$ and $c_i \in \mathbb{K}$ and $c_m \neq 0$. Let $\displaystyle\mathbb{E} := \bigcup_{n=1}^\infty \mathbb{R}((t^{1/n}))$ be the field of Puiseux series over $\mathbb{R}$. The following proposition follows from \cref{prop:CriterionToBeRealClosed} and \cite[Chapter IV, \S 2, Proposition 8]{SerreLocalFields}

\begin{prop}
\label{prop:EIsRealClosed}
    The field $\mathbb{E}$ of Puiseux series over $
\mathbb{R}$ is real closed. In particular, the field $\mathbb{E}$ admits a unique ordering whose positive cone is precisely the set of squares in $\mathbb{E}.$
\end{prop}

\begin{proof}
    We note that $\displaystyle\mathbb{E}(\sqrt{-1}) = \bigcup_{n=1}^\infty\mathbb{C}((t^{1/n}))$ and so $\mathbb{E}$ is a proper subfield of $\mathbb{E}(\sqrt{-1})$. By \cite[Chapter IV, \S 2, Proposition 8]{SerreLocalFields}, the field $\displaystyle\mathbb{E}(\sqrt{-1})$ is algebraically closed. Now it follows from \cref{prop:CriterionToBeRealClosed} that the field $\mathbb{E}$ is real closed. \cref{thm:ExistenceOfRealClosureAndProperties} tells us that the field $\mathbb{E}$ has a unique ordering defined by the positive cone consisting of precisely all elements of $\mathbb{E}$ that are square.   
\end{proof}

The fact that the field $\mathbb{E}$ is real closed and contains $\mathbb{Z}[t^{\pm 1}]$ is the main reason we are considering $\mathbb{E}$ to begin with. The unique ordering on $\mathbb{E}$ has a concrete description as follows. Consider the valuation $\deg_{\min}:\mathbb{E} \to \mathbb{Q} \cup \{\infty\}$ defined on $\mathbb{E}$ by setting 
\begin{equation}\label{eq:DegreeMin}
\deg_{\min}(0)=\infty \text{ and } \deg_{\min}(f) = r 
\end{equation}
where $f\neq0$ and $r$ is the smallest exponent of $t$ among all nonzero terms that appear in $f$. The function $\deg_{\min}$ is a valuation in the sense that it is a surjective group homomorphism from the group of units of $\mathbb{E}$ to $(\mathbb{Q},+)$ and also satisfies
\[
\deg_{\min}(f+g) \geq \min\{\deg_{\min}(f),\deg_{\min}(g)\},
\]
see \cite[Chapter 1, \S 1]{SerreLocalFields}. We define the function $\mathfrak{c}: \mathbb{E} \to \mathbb{R}$ setting $\frak{c}(0)=0$ and $\frak{c}(f)$ to be the coefficient of the term with the smallest exponent of $t$ for any $f \in \mathbb{E} \setminus \{0\}$. Using the function $\mathfrak{c}$, we define an ordering on $\mathbb{E}$. 

\begin{lem}
    \label{lem:OrderingOnE}
    The set $Q=\{f \mid \mathfrak{c}(f) > 0\}$ is closed under addition and multiplication. Furthermore, $\mathbb{E} = Q \sqcup \{0\} \sqcup -Q$. Therefore, $(\mathbb{E},Q)$ is an ordered field. 
\end{lem}

\begin{proof}
    Let $f,g \in Q$. Then we can write 
    \[
    f = \frak{c}(f)t^r + \dots \text{ and } g = \frak{c}(g)t^s  + \dots
    \]
    where $\frak{c}(f), \frak{c}(g) > 0$. The smallest nonzero term of $fg$ is $\mathfrak{c}(f) \mathfrak{c}(g)t^{r+s}$ which implies that $fg \in Q$. The smallest term of $f+g$ is either $\frak{c}(f)t^r $, $\frak{c}(g)t^s$, or $(\frak{c}(f) +\frak{c}(g))t^s$ since $\mathfrak{c}(f) + \mathfrak{c}(g) > 0$. In any case, we must have $f+g \in Q$. Now note that $\frak{c}(f) = 0 $ if and only if $f = 0$. Therefore, $Q$ defines an ordering on the field $\mathbb{E}$.
\end{proof}

We refer to the ordering on $\mathbb{E}$ defined in \cref{lem:OrderingOnE} as $Q$. In the subsequent section, we will need to consider the tensor power of $\mathbb{E}$ over $\mathbb{R}$. We denote by $\mathbb{E}_m := \mathbb{E}^{\otimes_\mathbb{R} m}$ the $m$-fold tensor power of $\mathbb{E}$ with coefficients over $\mathbb{R}$.

\begin{prop}
    \label{prop:OrderingOnTensorPowerOfPuiseuxField}
    Let $m \in \mathbb{N}$. There exists an ordering $Q_m$ on $\mathbb{E}_m$ such that $u_1\otimes \dots \otimes u_m \in  Q_m$ for any $ \{u_i\}\subset Q$.
\end{prop}

\begin{proof}[Proof of \cref{prop:OrderingOnTensorPowerOfPuiseuxField}]
    When $m =1$, we have $\mathbb{E}_1 = \mathbb{E}$ and can simply take $Q_1 = Q$ as the ordering on $\mathbb{E}$. Fix a positive integer $m\geq 2$. We first define a lexicographic ordering $\prec$ on the following set of simple tensors
    \[
    S_m = \{t^{q_1} \otimes \dots \otimes t^{q_m} \mid q_i \in \mathbb{Q}\}.
    \]
    In particular, we write  
    \[
    t^{q_1} \otimes \dots \otimes t^{q_m} \prec t^{q'_1} \otimes \dots \otimes t^{q'_m} \Leftrightarrow q_{j} < q'_j \text{ in } \mathbb{Q} \text{ and } q_i = q'_i \ \forall \ 1 \leq i < j. 
    \]
    \begin{lem}
        The set $S_m$ equipped with $\prec$ is an ordered multiplicative abelian group.
    \end{lem}
    \begin{proof}
        Indeed, $S_m$ is a multiplicative abelian group with the group operation given by 
        \[
        (t^{q_1} \otimes \dots \otimes t^{q_m})(t^{q'_1} \otimes \dots \otimes t^{q'_m}) = (t^{q_1+q'_1} \otimes \dots \otimes t^{q_m+q'_m}).
        \]
        To see that $\prec$ defines a total bi-invariant ordering on $S_m$, we consider
        \begin{align*}
    P_{\prec} := \{s\in S_m \mid 1_{S_m} \prec s\} = \{t^{q_1} \otimes \dots \otimes t^{q_m} \mid \exists 1\leq j \leq m, q_j >0 \text{ and } \forall 1 \leq i <j, q_i = 0 \}.    
    \end{align*}
    This set is closed under multiplication and partitions $S_m = P_{\prec} \sqcup \{1_{S_m}\}\sqcup P_{\prec}^{-1}$. Therefore, $\prec$ defines a total bi-invariant ordering on $S_m$ when viewed as a multiplicative abelian group. 
    \end{proof}
    
    Now, we will use the ordering $\prec$ on $S_m$ to define an ordering $Q_m$ for $\mathbb{E}_m$. First, we make an observation about simple tensors in $\mathbb{E}_m$.

    \begin{lem}
        Let $u = u_1 \otimes \dots \otimes u_m \in \mathbb{E}_m$ be a simple tensor where
        \[
        u_i =\sum\limits_{j=k_i}^{\infty} c_{ij}t^{j/d_i} \in \mathbb{E}
        \]
        Then $u$ is a formal $\mathbb{R}$-linear combination of elements in $S_m$. Furthermore, if $u\neq 0$, then the set of all elements in $S_m$ with non-zero coefficient in $u$ is well-ordered with respect to $\prec$. 
    \end{lem}

    \begin{proof}
        Let $d$ be the l.c.m of $d_1,\dots,d_m$. We can rewrite each $u_i$ as 
        \[
        u_i = \sum\limits_{j=k_i}^{\infty} c_{ij}t^{j/d_i} = \sum\limits_{\ell_i=k'_i}^{\infty} b_{i\ell_i}t^{\ell_i/d}
        \]
        where $\ell_i = jd/d_i$ and $b_{i\ell_i} = c_{ij}$. Now we can write 
        \begin{align*}
            u 
            &= \left(\sum\limits_{\ell_1=k'_1}^{\infty} b_{1\ell_1}t^{\ell_1/d}\right) \otimes \dots \otimes \left(\sum\limits_{\ell_m=k'_m}^{\infty} b_{m\ell_m}t^{\ell_m/d}\right) \\
            &= \sum\limits_{\ell_1=k'_1}^{\infty}\dots\sum\limits_{\ell_m=k'_m}^{\infty}\left(b_{1\ell_1}\dots b_{m\ell_m}\right)\left(t^{\ell_1/d} \otimes \dots \otimes t^{\ell_m/d}\right).
        \end{align*}
         The calculation shows that any simple tensor is a formal $\mathbb{R}$-linear combination of elements in $S_m$. The set of all elements in $S_m$ with non-zero coefficient in $u$ is contained in 
        \[
        \{ t^{\ell_1/d} \otimes \dots \otimes t^{\ell_m/d} \mid \ell_i \geq k'_i \text{ and } \ell_i \in \mathbb{Z}\}.
        \]       
        This set is order-isomorphic to $\mathbb{N}^m$ with the lexicographic order and, therefore, is well-ordered. In particular, $t^{k'_1/d} \otimes \dots \otimes t^{k'_m/d}$ is the smallest element in $S_m$ with respect to $\prec$ that has a non-zero coefficient in $u$. 
    \end{proof}

    In view of the previous lemma, we can assume that any simple tensor in $\mathbb{E}_m$ has the form 
    \[
    u = \sum\limits_{j_1 \geq k_1,\dots,j_m \geq k_m}^\infty c_{j_1,\dots,j_m} (t^{j_1/d} \otimes \dots \otimes t^{j_m/d})
    \]
    where $c_{k_1,\dots,k_m} \neq 0$. In particular, the formal sum runs over the set 
    \[
        \{ t^{j_1/d} \otimes \dots \otimes t^{j_m/d} \mid j_i \geq k_i \text{ and } j_i \in \mathbb{Z}\}
    \]
    which is well-ordered with respect to $\prec$. 
    \begin{lem}
    \label{lem:WellDefinedLowestTerm}
        Let $f \in \mathbb{E}_m$. Then $f$ is a formal $\mathbb{R}$-linear combination of elements in $S_m$. Furthermore, if $f\neq 0$, then the set of all elements in $S_m$ with non-zero coefficient in $f$ is well-ordered with respect to $\prec$. Consequently, for any $f \neq 0$, the coefficient of the lowest term, $\mathfrak{c}_m(f) \in \mathbb{R}$, of $f$ with respect to $\prec$ is well-defined. 
    \end{lem}

    \begin{proof}
        The proof of this is similar to the previous lemma. The point is that the set of denominators is still finite since $f$ is a finite $\mathbb{R}$-linear combination of simple tensors. So we can rewrite everything in terms of a single common denominator. In particular, every element of $\mathbb{E}_m$ can be expressed as a formal linear combination of terms in the set 
        \[
        \{ t^{j_1/d} \otimes \dots \otimes t^{j_m/d} \mid j_i \geq k_i \text{ and } j_i \in \mathbb{Z}\}
        \] 
        which is well-ordered. 
    \end{proof}

    Now we are in position to define the ordering $Q_m$ on $\mathbb{E}_m$. Given a nonzero element $f\in \mathbb{E}_m$, we express $f$ as a formal sum of terms in $S_m$ with a single denominator. By \cref{lem:WellDefinedLowestTerm}, this formal sum has a well-defined lowest term among all nonzero terms with respect to the ordering $\prec $ on $S_m$. We let $\mathfrak{c}_m(f) \in \mathbb{R}$ be the coefficient of the lowest nonzero term of $f$ with respect to the ordering $\prec$ of $S_m$. In addition, $\mathfrak{c}_m(0)= 0$. We define
    \[
    Q_m := \{f \in \mathbb{E}_m \mid \mathfrak{c}_m(f) > 0 \}. 
    \]
    Since every nonzero element of $\mathbb{E}_m$ has a well-defined lowest term with respect to $\prec$, we have 
    \[\mathbb{E}_m = Q_m \sqcup\{0\} \sqcup -Q_m.\] 
    Now, let $f,g \in Q_m$. Since the ordering $\prec$ on $S_m$ is bi-invariant, the lowest term of $fg$ is the product of the lowest terms of $f$ and $g$. This implies that $\frak{c}_m(fg) = \frak{c}_m(f)\frak{c}_m(g) >0$ and that $Q_m Q_m \subseteq Q_m$. The lowest term of $f+g$ is either the lowest term of $f$, $g$ or the sum of both since $\mathfrak{c}_m(f), \mathfrak{c}_m(g)$ are both positive. Therefore, $\frak{c}_m(f+g)$ is either $\frak{c}_m(f)$, $\frak{c}_m(g)$, or $\frak{c}_m(f)+\frak{c}_m(g)$. In any case, $\frak{c}_m(f+g) >0$ which implies that $Q_m + Q_m \subset Q_m$. Hence, $Q_m$ defines an ordering on $\mathbb{E}_m$. 

    For the claim that $u_1\otimes \dots \otimes u_m \in  Q_m$, for any $ \{u_i\}\subset Q$, we consider $u_i \in Q \subset \mathbb{E}$ for $1\leq i \leq m$. Note that $\mathfrak{c}_m(u_1\otimes \dots \otimes u_m) = \frak{c}(u_1)\dots\frak{c}(u_m) >0$ which implies that $u_1\otimes \dots \otimes u_m \in Q_m$. This completes the proof of \cref{prop:OrderingOnTensorPowerOfPuiseuxField}.
\end{proof}

By \cref{lem:OrderedRingIsAnID}, the ordered ring $(\mathbb{E}_m,Q_m)$ is an integral domain and therefore embeds in its field of fraction which admits a unique ordering extending that of $Q_m$. We denote the field of fractions of $\mathbb{E}_m$ by $\mathbb{F}_m$. Since the ordering on $\mathbb{F}_m$ extending $Q_m$ is unique, we shall use $Q_m$ to denote the positive cone that defines this ordering.  

\begin{remark}
    It is not true in general that the tensor product of ordered fields over an ordered subfield has to admit an ordering. Consider the field $\mathbb{Q}[x]/(x^2-2)$, which admits an ordering coming from an embedding to $\mathbb{R}$ induced by evaluating $x$ at $\sqrt{2}$. The multiplication 
    \[\mu : \mathbb{Q}[x]/(x^2-2) \otimes_\mathbb{Q}\mathbb{Q}[x]/(x^2-2) \to \mathbb{Q}[x]/(x^2-2)\] 
    defined by $\mu(u\otimes v) = uv$ is a $\mathbb{Q}$-linear map. Therefore, the map $\mu$ must have a nontrivial kernel for dimensional reasons. It follows that $\mathbb{Q}[x]/(x^2-2) \otimes_\mathbb{Q}\mathbb{Q}[x]/(x^2-2)$ is not an integral domain and does not admit any ordering. 
\end{remark}

\begin{remark}
    We suspect that $\mathbb{E}_m$ is already a field, that is, $\mathbb{E}_m = \mathbb{F}_m$ for all $m$. However, we do not need this in our subsequent arguments. 
\end{remark}

\subsection{Braids and the Burau representation}
In this section, we give a definition of the (reduced) Burau representation of braid groups. We refer the reader to \cite{kasselturaev08braid} for more details. In \cite[Section 3.2.2]{kasselturaev08braid}, the Burau representation of $B_n$ is defined as a certain twisted homological representation of the mapping class group of the $n$-punctured disk. The reduced Burau representation is obtained as a summand of this twisted homological representation, see \cite[Section 3.3]{kasselturaev08braid}.

Let $B_n$ be the braid group on $n$ strands. This group has a well-known presentation with $n-1$ generators $\sigma_i$ for $1\leq i \leq n-1$ and relations
\begin{align*}
    \sigma_i \sigma_j &= \sigma_j \sigma_i && \text{if } |i-j|>2, \\
    \sigma_i \sigma_{i+1} \sigma_i &= \sigma_{i+1}\sigma_i\sigma_{i+1} && \text{for } 1 \leq i \leq n-2
\end{align*}
Let $D_n$ be the disk with $n$ punctures which are labeled from $1$ to $n$. Let $\MCG(D_n,\partial D_n)$ be the group of homeomorphisms of $D_n$ fixing the boundary $\partial D_n$ pointwise modulo isotopy. It is well-known that there is an isomorphism between $B_n$ and $\MCG(D_n,\partial D_n)$ given by sending $\sigma_i$ to the Dehn half-twist between the $i$ and $(i+1)$ puncture on $D_n$. By fixing a basepoint $*$ on the boundary, we can identify the fundamental group of $D_n$ with the free group $F$ on $\{x_1,\dots,x_n\}$ where $x_i$ represents the loop based at $*$ that goes once clockwise around the $i^{th}$ puncture. By realizing $\sigma_i$ as a Dehn half-twists around the appropriate punctures, we get an induced action of $\sigma_i$ on $F \cong \pi_1(D_n,*)$. This is known as the Artin action of $B_n$ on $F$ given by 
\[
\sigma_i \cdot x_j =  
\begin{cases}
    x_ix_{i+1}x_i^{-1} & \text{ if } j = i \\
    x_{i} & \text{ if } j = i+1 \\
    x_j & \text{ if } |i-j| \geq 2
\end{cases}
\]
which induces a representation $\Psi:B_n \to \Out(F)$. The representation $\Psi$, up to conjugation, is the action of $B_n \cong \MCG(D_n,\partial D_n)$ on $F \cong \pi_1(D_n,*)$. 

Exploiting this relationship, one can define the (reduced) Burau representation as follows. Consider the following short exact sequence
\begin{equation}\label{eq:ExponentSumSES}
    1 \to K \to F \xrightarrow{\mu} \mathbb{Z} \to 0
\end{equation}
where $\mu(x_i)=1$ for all $1\leq i \leq n$. Since $\mu \circ \Psi(\beta) = \mu$ for all $\beta \in B_n$, the action of $B_n$ preserves the kernel $K$. Thus, we get an action of $B_n$ on $H_1(K;\mathbb{Z})$. Since $K$ is an infinitely generated free group, $H_1(K;\mathbb{Z})$ is a torsion-free abelian group of infinite rank. However, $H_1(K;\mathbb{Z})$ is finitely generated as a $\mathbb{Z}[t^{\pm 1}]$-module where the action of  $t$ is the action of the quotient $\mathbb{Z}$ on $K$ associated to the short exact sequence in \cref{eq:ExponentSumSES}. The following fact is well known. See \cite[Remark 3.11]{kasselturaev08braid}.

\begin{prop}
    \label{prop:HomologyModuleBurau}
    Let $H_{\mathbb{R}} = H_1(K;\mathbb{R})$. Then $H_{\mathbb{R}}$ has a structure of a free module of rank $n-1$ over $\mathbb{R}[t^{\pm 1}]$ with a basis given by
    \[\{ v_{i} \mid 1 \leq i \leq n-1\}\]
    where $v_{i} := [x_ix_{i+1}^{-1}] \in H_\mathbb{R}$. 
\end{prop}

The action of $B_n$ on $H_{\mathbb{R}}$ defines a representation of the braid group $B_n$ into $\Aut(H_\mathbb{R})$. Choosing the basis $v_i$ from \cref{prop:HomologyModuleBurau}, we get the (reduced) Burau representation $\rho: B_n \to \GL_{n-1}(\mathbb{R}[t^{\pm 1}])$:
\begin{equation}\label{eq:ReducedBurauRep}
    \rho(\sigma_1) = \begin{pmatrix}
        -t & 1 & 0 \\
        0 &  1 &0  \\
        0& 0& I_{n-3}
    \end{pmatrix}, \quad \rho(\sigma_i) = \begin{pmatrix}
        I_{i-2} & 0 & 0 & 0 & 0 \\
        0 & 1 & 0 & 0 & 0 \\
        0 & t & -t & 1 & 0 \\
        0 & 0 & 0 & 1 & 0 \\
        0 & 0 & 0 & 0 & I_{n-i-2} \\
    \end{pmatrix} \quad  \rho(\sigma_{n-1}) = \begin{pmatrix}
        I_{n-3} &0  & 0 \\
       0  &  1 & 0  \\
       0 & t & -t
    \end{pmatrix}
\end{equation}
where $2\leq i \leq n-2$. With respect to \cref{eq:ReducedBurauRep}, the elements of $H_\mathbb{R}$ are column vectors, and the action of $B_n$ on $H_\mathbb{R}$ is the usual matrix multiplication. 

We end this section by describing the module structure of $V_m := H_{\mathbb{R}}^{\otimes_\mathbb{R} m}$ over the ring obtained by taking the $m$-fold tensor power $R_m := \mathbb{R}[t^{\pm 1}]^{\otimes_\mathbb{R} m}$. For a tuple of elements $(r_1,\dots,r_m)$ and $(u_1,\dots,u_m)$ where $r_i \in \mathbb{R}[t^{\pm 1}]$ and $u_i \in H_\mathbb{R}$, the map on the $m$-fold Cartesian product of $H_\mathbb{R}$ defined by
\[
(u_1,\dots,u_m) \mapsto (r_1u_1,\dots,r_mu_m)
\]
is $\mathbb{R}$-multilinear since multiplication by $r \in \mathbb{R}[t^{\pm 1}]$ is an $\mathbb{R}$-linear map on $H_\mathbb{R}$. For every tuple $(r_1,\dots,r_m)$ with $r_i \in \mathbb{R}[t^{\pm 1}]$, we get a well-defined $\mathbb{R}$-linear map on $V_m$. For any $1 \leq i \leq m$ and $r_i, r'_i \in \mathbb{R}[t^{\pm 1}]$, we have
\begin{align*}
    &(r_1,\dots,r_i + r'_i, \dots, r_m) \cdot (u_1 \otimes \dots \otimes u_m) \\ 
    = &r_1 u_1 \otimes \dots \otimes(r_i + r'_i) u_i \otimes \dots \otimes r_m u_m \\ 
    = &(r_1,\dots,r_i, \dots, r_m)\cdot (u_1 \otimes \dots \otimes u_m) + (r_1,\dots,r'_i, \dots, r_m)  \cdot (u_1 \otimes \dots \otimes u_m). 
\end{align*}
For any $a \in \mathbb{R}$, $1 \leq i \leq n$ and $r_i \in \mathbb{R}[t^{\pm 1}]$, we have
\begin{align*}
    &a \ (r_1,\dots,r_m)\cdot  (u_1 \otimes \dots \otimes u_m) \\
    = &a \ (r_1u_1 \otimes \dots \otimes r_m u_m )\\
    = & \ (r_1u_1 \otimes \dots \otimes ar_i u_i \otimes \dots \otimes r_m u_m )\\
    = & (r_1,\dots,ar_i,\dots,r_m)\cdot (u_1 \otimes \dots \otimes u_m). 
\end{align*}
Therefore, we have a well-defined action of $R_m$ on $V_m$. This gives $V_m$ the structure of an $R_m$-module. By \cref{prop:HomologyModuleBurau}, the module $H_\mathbb{R}$ is free of rank $n-1$ with a free basis $\{v_i\}$ that corresponds to figure-8 loop around the $i$- and $i+1$-punctures. We show that this basis naturally gives a free basis for $V_m$ as an $R_m$-module.  

\begin{prop}
\label{prop:TensorPowerOfHomologyModuleIsFree}
    The $R_m$-module $V_m$ is a free $R_m$-module of rank $(n-1)^m$ with a basis given by $\{v_{i_1}\otimes \dots \otimes v_{i_m} \mid 1 \leq i_j \leq n-1 \}$. 
\end{prop}

\begin{proof}
    To see that the set $\{v_{i_1}\otimes \dots \otimes v_{i_m} \mid 1 \leq i_j \leq n-1 \}$ is a spanning set, it suffices to consider simple tensors in $V_m$. Let $u_1\otimes \dots \otimes u_m \in V_m$ where $\displaystyle u_j = \sum^{n-1}_{i=1} c_{i,j}v_{i}$. Then, we have 
    \begin{align*}
    u_1\otimes \dots \otimes u_m 
    &= \sum^{n-1}_{i=1} c_{i,1}v_{i}\otimes \dots \otimes\sum^{n-1}_{i=1} c_{i,m}v_{i} = \sum c_{i_1,1} v_{i_1}\otimes\dots\otimes c_{i_m,m}  v_{i_m} \\
    &=\sum (c_{i_1,1}\otimes\dots\otimes c_{i_m,m} )(v_{i_1} \otimes \dots \otimes v_{i_m}).
    \end{align*}
    Since the simple tensors are all $R_m$-linear combinations of $\{ v_{i_1} \otimes \dots \otimes v_{i_m} \mid 1 \leq i_j \leq m \}$, the set spans $V_m$. 
    
    For ease of notation, we write $R_1$ for $\mathbb{R}[t^{\pm 1}]$ and $V_1$ for $H_{\mathbb{R}}$. To show linear independence, we consider the $R_1$-module maps $\phi_j: V_1 \to R_1$ defined on the basis $\{v_i \mid 1 \leq i \leq n-1\}$ by
    \[
    \phi_j(v_i) = \begin{cases}
        1 & \text{ if } i = j \\
        0 & \text{ if } i \neq j
    \end{cases}.
    \]
    Composing the direct product of $\phi_i$'s and the canonical projection $R_1 \times \dots \times R_1 \to R_m$, we obtain the maps $\phi_{i_1,\dots,i_m}:V_1 \times \dots \times V_1 \to R_m$. We have

    \begin{align*}
    \phi_{i_1,\dots,i_m}(u_1,\dots,u_j + u'_j,\dots,u_m) 
    &= \phi_{i_1}(u_1) \otimes \dots \otimes \phi_{i_j}(u_j + u'_j) \otimes \dots \otimes \phi_{i_m}(u_m) \\ 
    &= \phi_{i_1}(u_1) \otimes \dots \otimes (\phi_{i_j}(u_j) + \phi_{i_j}(u'_j)) \otimes \dots \otimes \phi_{i_m}(u_m) \\ 
    & = \phi_{i_1,\dots,i_m}(u_1,\dots,u_j,\dots,u_m) + \phi_{i_1,\dots,i_m}(u_1,\dots,u'_j,\dots,u_m), \text{ and } \\
    \phi_{i_1,\dots,i_m}(u_1,\dots,r u_j,\dots,u_m) 
    &= \phi_{i_1}(u_1) \otimes \dots \otimes r\phi_{i_j}(u_j) \otimes \dots \otimes \phi_{i_m}(u_m) \\   
    &= r(\phi_{i_1}(u_1) \otimes \dots \otimes \phi_{i_j}(u_j) \otimes \dots \otimes \phi_{i_m}(u_m)) \\
    \end{align*}
    which implies that $\phi_{i_1,\dots,i_m}$ is $\mathbb{R}$-multilinear.
    We get a well-defined map $\overline{\phi}_{i_1,\dots,i_m}:V_1 \otimes \dots \otimes V_1 \to R_m$. In fact, $\overline{\phi}_{i_1,\dots,i_m}$ is an $R_m$-module map which we can verify by checking the simple tensors:
    \begin{align*}
    \overline{\phi}_{i_1,\dots,i_m}((f_1\otimes \dots \otimes f_m)(u_1\otimes \dots \otimes u_m)) 
    &= \overline{\phi}_{i_1,\dots,i_m}((f_1u_1\otimes \dots \otimes f_mu_m)) \\
    &= \phi_{i_1}(f_1u_1)\otimes \dots \otimes \phi_{i_m}(f_mu_m) \\
    &= f_1\phi_{i_1}(u_1)\otimes \dots \otimes f_m\phi_{i_m}(u_m) \\
    &= (f_1\otimes \dots \otimes f_m)(\phi_{i_1}(u_1)\otimes \dots \otimes \phi_{i_m}(u_m)) \\
    &= (f_1\otimes \dots \otimes f_m)\overline{\phi}_{i_1,\dots,i_m}(u_1\otimes \dots \otimes u_m)
    \end{align*}
    where $f_i \in R_1$ and $u_i \in V_1$. On the set  
    \[
    \{ v_{i_1} \otimes \dots \otimes v_{i_m} \mid 1 \leq i_j \leq m \},
    \]
    the $R_m$-module map $\overline{\phi}_{i_1,\dots,i_m}$ vanishes everywhere except
    \[
    \overline{\phi}_{i_1,\dots,i_m}(v_{i_1} \otimes \dots \otimes v_{i_m}) = 1_{R_m}
    \]
    This implies that $\{ v_{i_1} \otimes \dots \otimes v_{i_m} \mid 1 \leq i_j \leq m \}$ is $R_m$-linearly independent. Therefore, $V_m$ is a free $R_m$-module of rank $(n-1)^m$.
\end{proof}

\begin{cor}
    \label{cor:ExtensionOfScalarsForV_m}
    The $\mathbb{E}_m$-module $V_m \otimes_{R_m} \mathbb{E}_m$ is free of rank $(n-1)^m$. The $\mathbb{F}_m$-vector space $V_m \otimes_{R_m} \mathbb{F}_m$ has dimension $(n-1)^m$. Furthermore, the containment $R_m \subset \mathbb{E}_m \subset \mathbb{F}_m $ induces injective maps $V_m \hookrightarrow V_m \otimes_{R_m} \mathbb{E}_m \hookrightarrow V_m \otimes_{R_m} \mathbb{F}_m$ given by
    \[
    v \longmapsto v\otimes 1 \longmapsto v \otimes 1.
    \]
\end{cor}

\begin{proof}
    By \cref{prop:TensorPowerOfHomologyModuleIsFree}, $V_m$ is a free $R_m$-module of rank $(n-1)^m$. That is $V_m$ is isomorphic to $\bigoplus\limits_{i=1}^{(n-1)^m}R_m$ as an $R_m$-module. It follows that the $\mathbb{E}_m$-module $V_m \otimes_{R_m} \mathbb{E}_m$ is isomorphic to  $\bigoplus\limits_{i=1}^{(n-1)^m}\mathbb{E}_m$ and is free of rank $(n-1)^m$. Similarly, the $\mathbb{F}_m$-vector space $V_m \otimes_{R_m} \mathbb{F}_m$ is isomorphic to a $\mathbb{F}_m$-vector space, $\bigoplus\limits_{i=1}^{(n-1)^m}\mathbb{F}_m$, of dimension $(n-1)^m$. As a consequence, the containment $R_m \subset \mathbb{E}_m \subset \mathbb{F}_m $ induces injective maps $V_m \hookrightarrow V_m \otimes_{R_m} \mathbb{E}_m \hookrightarrow V_m \otimes_{R_m} \mathbb{F}_m$ given by
    \[
    v \longmapsto v\otimes 1 \longmapsto v \otimes 1.
    \]
    via extension of scalars. 
\end{proof}

\section{Order-preserving braids via the Burau representation}
\subsection{Proof of the main theorem}
\label{subsec:ProofMainTheorem}
Let $F$ be the free group on $\{x_1,\dots, x_n\}$ and consider the short exact sequence
\[
1 \to K \to F \xrightarrow{\mu} \mathbb{Z} \to 0
\]
where $\mu(x_i) = 1$ for all $1\leq i \leq n$.

The following is a variation of \cite[Lemma 3.1]{CaiClayRolfsen}. We start by describing some group homomorphisms used in the lemma. First, we need a few definitions from Fox's free differential calculus \cite{fox1953freedifferentialcalculus}. Consider the group ring $\mathbb{Z} K$ of $K$ over $\mathbb{Z}$, and $\varepsilon:\mathbb{Z}K \to \mathbb{Z}$ the \emph{augmentation homomorphism} of the group ring defined by  
\[
\varepsilon\left(\sum_{i=1}^k c_i g_i \right) = \sum_{i=1}^k c_i.
\]
Let $I = \ker(\varepsilon)$ be an ideal of $\mathbb{Z} K$ and $I^m$ be the power of the ideal $I$ in the group ring $\mathbb{Z}K$. We refer to the ideals $\{I^m\}$ as the \emph{lower central ideals} following \cite{fox1953freedifferentialcalculus}. For any $g \in K = K_1$, we have $g \in K_m$ if and only if $g-1 \in I^m$, see \cite[Proposition 4.6]{fox1953freedifferentialcalculus}. We get an injective group homomorphism $\zeta_m: K_m /K_{m+1} \to I^m /I^{m+1}$ given by 
\[\
\zeta_m([g]) = [g-1],
\] 
where square brackets denote the equivalence class of an element in the appropriate quotient. For a free basis $\{z_i \mid i \in \mathbb{Z}_{>0}\}$ of $K$, the free abelian group $I^m /I^{m+1}$ admits a free $\mathbb{Z}$-basis 
\begin{equation}
\label{eq:FreeZBasisForLowerCentralQuotient}
    \{[(z_{i_1}-1)\dots (z_{i_m}-1)] \mid i_1,\dots,i_m \in \mathbb{Z}_{>0}\},
\end{equation} 
see \cite[Proposition 4.7]{fox1953freedifferentialcalculus}. The $m$-fold tensor power $H_1(K;\mathbb{Z})^{\otimes_\mathbb{Z} m }$ admits the following free $\mathbb{Z}$-basis $\{[z_{i_1}] \otimes_{\mathbb{Z}}\dots\otimes_{\mathbb{Z}}[ z_{i_m}]\}$. The assignment 
\[
[(z_{i_1}-1)\dots (z_{i_m}-1)] \mapsto [z_{i_1}] \otimes_{\mathbb{Z}} \dots \otimes_{\mathbb{Z}} [z_{i_m}]
\]
extends to an isomorphism of abelian groups $\eta_m:I^m/I^{m+1} \to H_1(K;\mathbb{Z})^{\otimes_\mathbb{Z} m }$. Similarly, square brackets denote the equivalence class of an element in the appropriate quotient. Extending the coefficients to $\mathbb{R}$, we have a canonical embedding $H_1(K;\mathbb{Z}) \to H_\mathbb{R}$. This embedding extends to an embedding of the $m$-fold tensor power $\theta_m: H_1(K;\mathbb{Z})^{\otimes_{\mathbb{Z}} m}\to H_\mathbb{R} ^{\otimes_{\mathbb{R}} m}$. On a basis for $H_1(K;\mathbb{Z})^{\otimes_{\mathbb{Z}} m}$, we define $\theta_m$ as
\[
\theta_m([z_{i_1}] \otimes_{\mathbb{Z}} \dots \otimes_{\mathbb{Z}} [z_{i_m}]) = [z_{i_1}] \otimes_{\mathbb{R}} \dots \otimes_{\mathbb{R}} [z_{i_m}].
\]
Recall that $V_m = H_\mathbb{R}^{\otimes_{\mathbb{R}} m}$. In this notation, \cref{cor:ExtensionOfScalarsForV_m} provides us with an embedding $\iota_m: H_\mathbb{R}^{\otimes_{\mathbb{R}} m } \to V_m \otimes_{R_m} \mathbb{F}_m$. In summary, we have the following sequence  
\begin{equation}
    \label{eq:HorizontalMaps}
    \begin{tikzcd}
     K_m/K_{m+1} \arrow[hookrightarrow]{r}{\zeta_m} & 
     I^{m}/I^{m+1} \arrow[r,"\eta_m"] &
     H_1(K;\mathbb{Z})^{\otimes_{\mathbb{Z}} m} \arrow[hookrightarrow]{r}{\theta_m} & 
     H_\mathbb{R}^{\otimes_{\mathbb{R}} m} \arrow[hookrightarrow]{r}{\iota_m}&
     V_m \otimes_{R_m} \mathbb{F}_m
    \end{tikzcd}
\end{equation}
where every map in the sequence is injective.

Now we consider an automorphism $\varphi \in \Aut(F)$ such that $\mu = \mu \circ \varphi$. Since $\mu = \mu \circ \varphi$, the automorphism $\varphi$ preserves the subgroup $K = \ker(\mu)$. The subgroup $K_m$ is a characteristic subgroup of $K$ and is preserved by $\varphi$. Therefore, we have an induced automorphism $\varphi_m:K_m/K_{m+1} \to K_m/K_{m+1}$. We also have an induced automorphism $\varphi_m':I^{m}/I^{m+1} \to I^{m}/I^{m+1}$ on each quotient of the lower central ideals. More concretely, using a free $\mathbb{Z}$-basis for $I^m/I^{m+1}$ as in \cref{eq:FreeZBasisForLowerCentralQuotient}, we can define $\varphi_m'$ as 
\[
\varphi_m'( [(z_{i_1}-1)\dots (z_{i_m}-1)]) = [(\varphi(z_{i_1})-1)\dots (\varphi(z_{i_m})-1)]. 
\]
When $m=1$, the automorphism $\varphi_1$ is the abelianization $\varphi_{\ab}:H_1(K;\mathbb{Z}) \to H_1(K;\mathbb{Z})$. Abusing notation, we also use $\varphi_{\ab}$ to denote the induced automorphism on $H_\mathbb{R}$. As argued in \cref{eq:CheckingVarphiAndTauCommute}, the action of $t$ commutes with the action of $\varphi_{ab}$ on $H_{\mathbb{R}}$ since $\mu = \mu \circ \varphi$. Therefore, $\varphi_{ab}:H_\mathbb{R} \to H_\mathbb{R}$ is also an $\mathbb{R}[t^{\pm 1}]$-module map. The $m$-fold tensor power of $\varphi_{ab}$ is an $R_m$-module map $\varphi_{ab}^{\otimes_{\mathbb{R}} m}: V_m \to V_m$ defined by the formula
\[
\varphi_{ab}^{\otimes_{\mathbb{R}} m}(u_{i_1}\otimes \dots \otimes u_{i_m}) = \varphi_{ab}(u_{i_1})\otimes \dots \otimes \varphi_{ab}(u_{i_m})
\]
on simple tensors and is extended by linearity to $V_m$. By extension of scalars, we get a $\mathbb{F}_m$-linear map $\varphi^{\otimes_\mathbb{R} m}_{ab} \otimes_{R_m} \Id: V_m \otimes_{R_m}\mathbb{F}_m \to V_m \otimes_{R_m} \mathbb{F}_m$. Putting together \cref{eq:HorizontalMaps} and the induced maps by $\varphi$, we have the following lemma.

\begin{lem}
    \label{lem:EmbeddingTheLowerCentralQuotients}
     Let $\varphi \in \Aut(F)$ be an automorphism such that $\pi = \pi \circ \varphi$. Then the following diagram commutes:
     \begin{equation}
     \label{eq:DiagramOfLowCentralSeriesQuotients}
     \begin{tikzcd}
        K_m/K_{m+1} \arrow[hookrightarrow]{r}{\zeta_m} \arrow[d,"\varphi_m"] 
        & I^m/I^{m+1} \arrow{r}{\eta_m} \arrow[d,"\varphi'_m"] 
        & H_1(K;\mathbb{Z})^{\otimes_\mathbb{Z} m }\arrow[hookrightarrow]{r}{\theta_m}\arrow[d,"\varphi_{\ab}^{\otimes_\mathbb{Z} m}"] 
        & H_\mathbb{R}^{\otimes_\mathbb{R} m } \arrow[hookrightarrow]{r}{\iota_m} \arrow[d, "\varphi_{\ab}^{\otimes_\mathbb{R} m}"] & V_m \otimes_{R_m} \mathbb{F}_m \arrow[d, "\varphi_{\ab}^{\otimes_{\mathbb{R}} m}\otimes_{R_m}\Id"] \\
        K_m/K_{m+1} \arrow[hookrightarrow]{r}{\zeta_m} 
        & I^m/I^{m+1} \arrow{r}{\eta_m} 
        &
        H_1(K;\mathbb{Z})^{\otimes_\mathbb{Z} m }\arrow[hookrightarrow]{r}{\theta_m} 
        &
        H_\mathbb{R}^{\otimes_\mathbb{R} m} \arrow[hookrightarrow]{r}{\iota_m} & V_m \otimes_{R_m} \mathbb{F}_m  
     \end{tikzcd}
     \end{equation}
     Furthermore, the quotient $K_m/K_{m+1}$ is mapped injectively into $V_m \otimes_{R_m} \mathbb{F}_m$ by composing all horizontal homomorphisms.  
 \end{lem}

 \begin{proof}
    We recall that in \cite{CaiClayRolfsen}, the authors define $H := H_1(K;\mathbb{Z})$. In particular, the first two squares of \cref{eq:DiagramOfLowCentralSeriesQuotients} are identical to the first two squares of the commutative diagram in \cite[Lemma 3.1]{CaiClayRolfsen}. The commutativity of the first two squares of \cref{eq:DiagramOfLowCentralSeriesQuotients} was proved in \cite[Lemma 3.1]{CaiClayRolfsen}. 

    The commutativity of the last two squares of \cref{eq:DiagramOfLowCentralSeriesQuotients} follows from the fact that the maps involve are canonical. We demonstrate this by direct calculations. We choose a free basis for $K$ denoted as $\{z_i \mid i \in \mathbb{Z}_{>0}\}$. We have
    \begin{align*}
        \varphi_{\ab}^{\otimes_{\mathbb{R}} m}(\theta_m(\sum c_{i_1,\dots,i_m}[z_{i_1}] \otimes_{\mathbb{Z}} \dots \otimes_{\mathbb{Z}}  [z_{i_m}])) 
        &= \varphi_{\ab}^{\otimes_{\mathbb{R}} m}(\sum c_{i_1,\dots,i_m}[z_{i_1}] \otimes_{\mathbb{R}} \dots \otimes_{\mathbb{R}}  [z_{i_m}]) \\
        &= \sum c_{i_1,\dots,i_m}\varphi_{\ab}([z_{i_1}]) \otimes_{\mathbb{R}} \dots \otimes_{\mathbb{R}}  \varphi_{\ab}([z_{i_m}]) \\
        &= \theta_m(\sum c_{i_1,\dots,i_m}\varphi_{\ab}([z_{i_1}]) \otimes_{\mathbb{Z}} \dots \otimes_{\mathbb{Z}}  \varphi_{\ab}([z_{i_m}])) \\
        &= \theta_m(\varphi_{\ab}^{\otimes_{\mathbb{Z}}m}(\sum c_{i_1,\dots,i_m}[z_{i_1}] \otimes_{\mathbb{Z}} \dots \otimes_{\mathbb{Z}}  [z_{i_m}])) \\
    \end{align*}
    where all the sums involved have finitely many non-zero terms each with coefficient in $\mathbb{Z}$. This gives commutativity of the third square of \cref{eq:DiagramOfLowCentralSeriesQuotients}. By another straightforward computation, we get
    \[
    (\varphi_{\ab}^{\otimes_{\mathbb{R}} m} \otimes \Id)(\iota_m(v)) = (\varphi_{\ab}^{\otimes_{\mathbb{R}} m} \otimes \Id)(v \otimes_{R_m} 1) = \varphi_{\ab}^{\otimes_{\mathbb{R}} m} (v) \otimes 1 = \iota_m (\varphi_{\ab}^{\otimes_{\mathbb{R}} m}(v)).
    \]
    which shows the commutativity of the right-most square. The quotient $K_m/K_{m+1}$ is mapped injectively into $V_m \otimes_{R_m} \mathbb{F}_m$ since all the horizontal maps are injective. 
 \end{proof}

Before proving \cref{thm:OPViaBurauRep}, we will need the following ingredient.

\begin{lem}\label{lem:TensorProductOfEigenvalues}
     Suppose that $V$ is a free $R$-module of rank $n-1$. Let $\varphi: V\to V$ be an $R$-linear map. By extending the scalars from $R$ to $\mathbb{E}$, we also view $\varphi: V \to V$ as a $\mathbb{E}$-linear map. If all of the eigenvalues $\{\lambda_1,\dots,\lambda_{n-1}\}$, counted with multiplicity, of the $\mathbb{E}$-linear map $\varphi$ are elements of $\mathbb{E}$, then all the eigenvalues of the $\mathbb{F}_m$-linear map $\varphi^{\otimes_\mathbb{R} m} \otimes_{R_m} \Id : V^{\otimes_{\mathbb{R}} m} \otimes_{R_m} \mathbb{F}_m \to V^{\otimes_{\mathbb{R}} m} \otimes_{R_m} \mathbb{F}_m$ are of the form
     \[
     \lambda_{i_1} \otimes_{\mathbb{R}} \dots \otimes_{\mathbb{R}}\lambda_{i_m} 
     \]
     where $1 \leq i_1,\dots,i_m \leq n-1$ for any $m$.
\end{lem}

\begin{proof}
    Since all of the eigenvalues of $\varphi$ are in $\mathbb{E}$, we can find an ordered $\mathbb{E}$-basis $\{u_1,\dots,u_{n-1}\}$ for $V$ such that
    \[
    \varphi(u_i) = \lambda_i u_i + \sum_{j=1}^{i-1} c_{ij} u_j
     \]
     where $c_{ij} \in \mathbb{E}$. The proof of \cref{prop:TensorPowerOfHomologyModuleIsFree} shows that the set 
     \[\{ u_{i_1} \otimes_{\mathbb{R}} \dots \otimes_{\mathbb{R}} u_{i_m} \otimes_{R_m} 1\mid 1 \leq i_1,\dots,i_m \leq n-1\}\] 
     forms a basis for $V_m\otimes_{R_m} \mathbb{F}_m$. We put a lexicographic ordering on this finite set of basis elements by defining 
     
     \[
     u_{i_1} \otimes_{\mathbb{R}} \dots \otimes_{\mathbb{R}} u_{i_m}\otimes_{R_m} 1 \prec u_{i'_1} \otimes_{\mathbb{R}} \dots \otimes_{\mathbb{R}} u_{i'_m}\otimes_{R_m} 1
     \]
     if and only if $i_j < i'_j$ and $i_k = i'_k$ for all $1\leq k \leq j-1$. Now we have
     
     \begin{align*}
     &(\varphi^{\otimes_\mathbb{R} m} \otimes_{R_m} \Id) (u_{i_1} \otimes_{\mathbb{R}} \dots \otimes_{\mathbb{R}} u_{i_{m}}\otimes_{R_m} 1) \\
     = &(\lambda_{i_1} u_{i_1} + \sum_{j_1=1}^{i_1-1} c_{i_1j_1} u_{j_1}) \otimes_{\mathbb{R}} \dots \otimes_{\mathbb{R}} (\lambda_{i_m} u_{i_m} + \sum_{j_m=1}^{i_m-1} c_{i_mj_m} u_{j_m})\otimes_{R_m} 1.  \\    
     \end{align*}
    \noindent
    We note that the coefficient of the largest term with respect to $\prec$ is 
    \[
    \lambda_{i_1} \otimes_{\mathbb{R}} \dots \otimes_{\mathbb{R}}\lambda_{i_m}\otimes_{R_m} 1 = \lambda_{i_1} \otimes_{\mathbb{R}} \dots \otimes_{\mathbb{R}}\lambda_{i_m}.
    \]
    With respect to the ordered $\mathbb{F}_m$-basis $\{ u_{i_1} \otimes \dots \otimes u_{i_m} \otimes 1 \mid 1 \leq i_1,\dots,i_m \leq n-1\}$ and the lexicographic ordering $\prec$, the map $\varphi^{\otimes_\mathbb{R} m} \otimes_{R_m} \Id$ can be represented as an upper-triangular matrix over $\mathbb{E}_m$ with $\lambda_{i_1} \otimes_{\mathbb{R}} \dots \otimes_{\mathbb{R}}\lambda_{i_m}$ on the diagonal. This proves the lemma.
\end{proof}

\begin{proof}[Proof of \cref{thm:OPViaBurauRep}]
    Let $F$ be the free group on $\{x_1,\dots, x_n\}$ and consider the short exact sequence
    \[
    1 \to K \to F \xrightarrow{\mu} \mathbb{Z} \to 0
    \]
    where $\mu(x_i) = 1$ for all $1\leq i \leq n$. The Artin action of any braid $\beta \in B_n$ on $F$ preserves $\mu.$ Therefore, the automorphism $\beta$ induces a well-defined map $\beta|_K: K \to K$. 

    We claim that if $\beta|_K$ preserves a positive cone $P$ on $K$, then $\beta$ preserves a positive cone on $F$. Indeed the positive cone on $F$ can be defined by 
    \[
    P \cup \mu^{-1}(\{k \in \mathbb{Z} \mid k >0\}). 
    \]
    Furthermore, if the action of $\mathbb{Z}$ on $K$ preserves the same positive cone $P$, then the positive cone $P \cup \mu^{-1}(\{k \in \mathbb{Z} \mid k >0\})$ defines a bi-ordering on $F$. In particular, to show that $\beta$ is order-preserving, it suffices for us to construct a conjugation-invariant positive cone on $K$ preserved by $\beta|_K$ and $x_i$ for any $i$ and hence for all $1 \leq i\leq n$. We now focus on the action of $\beta$ on $K$, and with a slight abuse of notation, write $\beta$ for $\beta|_K$. 
    
    We consider the lower central series of $K$ defined inductively by setting $K_1 = K$ and $K_m = [K,K_{m-1}]$. The automorphism $\beta$ induces an automophism $\beta_m:K_m/K_{m+1} \to K_m/K_{m+1}$ on the quotients of the lower central series. By \cref{lem:EmbeddingTheLowerCentralQuotients}, each quotient $K_m/K_{m+1}$ embeds into the vector space $V_m \otimes_\mathbb{R} \mathbb{F}_m $ where $\mathbb{F}_m$ is an ordered field when equipped with the positive cone $Q_m$. The automorphism $\beta_m$ extends to a $\mathbb{F}_m$-linear automorphism $\beta_{ab}^{'\otimes_\mathbb{R} m}$ where $\beta_{ab}^{'}$ is the induced map on the abelianization of $K$ with coefficients extended to the field $\mathbb{E}$ of Puiseux series over $\mathbb{R}$. 
    
    By the definition of the Burau representation, the map $\beta_{ab}^{'}$ is the image of $\beta$ under the Burau representation with an appropriate choice of basis. By our hypothesis, all of the eigenvalues of the linear map $\beta_{ab}^{'}$ lie in $Q \subset \mathbb{E}$. For any $m$, all of the eigenvalues of the map $\beta_{ab}^{'\otimes_\mathbb{R} m}$ are tensor product of the eigenvalues of $\beta_{ab}^{'}$ by \cref{lem:TensorProductOfEigenvalues}. Therefore, all of the eigenvalues of the map $\beta_{ab}^{'\otimes_\mathbb{R} m}$ lie in $Q_m$ since they are simple tensor of elements in $Q$. \cref{prop:PositiveEigenvaluesImpliesOP} implies that there exists an ordering on $V_m \otimes_\mathbb{R} \mathbb{F}_m$ preserved by $\beta_{ab}^{'\otimes_\mathbb{R} m}$. Restricting this ordering to $K_m/K_{m+1}$, we obtained an ordering $P_m$ on $K_m/K_{m+1}$ that $\beta_m$ preserves. Note that the induced action of $t$ also preserves all of these orderings since the action of $t$ is multiplying by the positive scalar $t^{\otimes m} \in Q_m \subset \mathbb{F}_m$. 

    The positive cones $P_m$ on $K_m/K_{m+1}$ can be used to obtain an invariant ordering on $K$ as follows. It is well known that $\cap_{m=1}^\infty K_m = \{1\}$. Therefore, for any nontrivial element $g \in K$, there exists a largest $m = m(g)>0$ such that $[g]$ is nontrivial in $K_m/K_{m+1}$. We say that $g \in P$ if and only if $g \in P_{m(g)}$. Since $P_m$ is invariant under the induced action of $\beta$ and $t$, the set $P$ is invariant under the action of $\beta$ and $t$. This implies that $\beta$ is order-preserving. 
\end{proof}

\subsection{Applications to order-preserving 3-braids}
\label{subsec:OP3Braid}
We recall that for any $f \in \mathbb{E}$, we defined $\deg_{\min}(f) \in \mathbb{Q}$ to be the smallest exponent of $t$ among all non-zero terms of $f$, see \cref{eq:DegreeMin}. Furthermore, the (reduced) Burau representation is given by \cref{eq:ReducedBurauRep}.
    \[
    \rho(\sigma_1) = \begin{pmatrix}
        -t & 1 \\ 0 & 1
    \end{pmatrix} 
    \quad \text{and} \quad 
    \rho(\sigma_2) = \begin{pmatrix}
        1 & 0 \\ t & -t
    \end{pmatrix}
    \]

Now, we perform some preliminary calculations before proving \cref{thm:EvenEvenBraidsAreOP}.

\begin{lem}
    \label{lem:BaseCase}
    For any $a \in \mathbb{N}$, we have
    \[
    \rho(\sigma_1\sigma_2^{-a}) = \begin{pmatrix}
       f_a - t & (-t)^{-a} \\f_a  & (-t)^{-a}
    \end{pmatrix}
    \]
    where $f_a = \sum\limits_{i=0}^{a-1}(-t)^{-i}$, $a \in \mathbb{N},$ and $f_{0} = 0$ by convention. In particular, $\deg_{\min}(\tr(\rho(\sigma_1\sigma_2^{-a}))) = -a$.
\end{lem}

\begin{proof}
    We proceed by induction on $a$. When $a = 0$, the formula in \cref{lem:BaseCase} is the Burau representation $\rho(\sigma_1)$ defined in \cref{eq:ReducedBurauRep} where $f_0 = 0$ by convention. For the inductive step, we check that
    \[
    \rho(\sigma_1\sigma_2^{-a-1}) =  \begin{pmatrix}
        f_a - t & (-t)^{-a}\\f_a  & (-t)^{-a}
    \end{pmatrix}\begin{pmatrix}
        1 & 0 \\ 1 & -t^{-1}
    \end{pmatrix} = \begin{pmatrix}
        f_{a+1} - t & (-t)^{-a-1}\\ f_{a+1}  & (-t)^{-a-1}
    \end{pmatrix}.
    \] 
    This completes the induction. From the formula, we see that $\deg_{\min}(\tr(\rho(\sigma_1\sigma_2^{-a}))) = -a$.
\end{proof}

\begin{prop}
    \label{prop:PositiveDiscriminant}
    Let $\beta =  \sigma_1\sigma_2^{-a_1}\dots   \sigma_1\sigma_2^{-a_k}$ where $a_i \geq 0$, $k\geq 1$ and denote
    \[
    \rho(\beta) = \begin{pmatrix} \beta_{11} & \beta_{12} \\ \beta_{21} & \beta_{22} \end{pmatrix}
    \]
    Then we have 
    \begin{equation}
        \begin{aligned}
        \label{eq:MinDegreeTrace}
        \deg_{\min}(\beta_{12}) &=\deg_{\min}(\beta_{22}) =   \deg_{\min}(\beta_{11} + \beta_{12}) = \deg_{\min}(\beta_{21} + \beta_{22}) = -a_1 -\dots - a_k  \\
    \end{aligned}
    \end{equation}
    Furthermore, $\deg_{\min}(\beta_{22}) < \deg_{\min}(\beta_{11})$, $\deg_{\min}(\beta_{22}) < \deg_{\min}(\beta_{21})$ if $\beta_{21} \neq 0$. The smallest nonzero coefficient of $\tr(\rho(\beta))$ is $\mathfrak{c}(\tr(\rho(\beta)) = (-1)^{-a_1-\dots-a_k}$. Consequently, $\tr(\rho(\beta))^2 - 4\det(\rho(\beta))$ is positive in $(\mathbb{E},Q)$ and that $\tr(\rho(\beta))$ is positive in $(\mathbb{E},Q)$ if and only if $-a_1-\dots-a_k$ is even.   
\end{prop}

\begin{proof}
    We verify the proposition by inducting on $k$. Using \cref{lem:BaseCase} and letting $\beta =\sigma_1\sigma_2^{-a}$, we have
    \[
    \rho(\beta) = \begin{pmatrix}
        f_a - t & (-t)^{-a} \\ f_a  & (-t)^{-a}
    \end{pmatrix}
    \]
    where $f_a = \sum\limits_{i=0}^{a-1}(-t)^{-i}$ which verifies all the claims about the minimal degrees of $\rho(\beta)$ in \cref{eq:MinDegreeTrace} for any $a \geq 0$. We have $\deg_{\min}(\beta_{11}) = -a + 1 > \deg_{\min}(\beta_{22})$ for all $a \geq 0$. Therefore, the coefficient $\mathfrak{c}(\tr(\rho(\beta)))$ is $(-1)^{-a}$. The trace $\tr(\rho(\beta))$ is positive in $(\mathbb{E},Q)$ if and only if $-a$ is even. Comparing the minimum degree of the terms in the discriminant of $\rho(\beta)$, we get 
    \[
    \deg_{\min}(\tr(\rho(\beta))^2) = -2a < -a + 1 = \deg_{\min}(\det(\rho(\beta)))\] 
    for all $a \geq 0$. Therefore, the discriminant $\tr(\rho(\beta))^2 - 4\det(\rho(\beta))$ is positive in $(\mathbb{E},Q)$. These calculations verify the base case. 

    Suppose that the proposition is true for $k-1$. Let $\beta = \sigma_1\sigma_2^{-a_1}\dots\sigma_1\sigma_2^{-a_k} $ and $\beta' = \sigma_1\sigma_2^{-a_1}\dots\sigma_1\sigma_2^{-a_{k-1}} $. We have 
    \begin{align*}
    \rho(\beta) &= \rho(\beta')\rho(\sigma_1\sigma_2^{-a_k} ) = 
    \begin{pmatrix}
        \beta'_{11} & \beta'_{12} \\ \beta'_{21} & \beta'_{22} 
    \end{pmatrix}
     \begin{pmatrix}
        f_{a_{k}} - t & (-t)^{-a_k}  \\ f_{a_{k}} &  (-t)^{-a_k} \\ 
    \end{pmatrix}
     \\ 
    &=\begin{pmatrix}
    (\beta'_{11}+\beta'_{12})f_{a_k} - t\beta'_{11} & (\beta'_{11}+\beta'_{12})(-t)^{-a_k} \\ (\beta'_{21}+\beta'_{22})f_{a_k} - t\beta'_{21} & (\beta'_{21}+\beta'_{22})(-t)^{-a_k}
    \end{pmatrix} 
    \end{align*}
    The induction hypothesis $\deg_{\min}((\beta'_{11}+\beta'_{12}))= \deg_{\min}((\beta'_{21}+\beta'_{22})) = -a_1 -\dots -a_{k-1}$ implies that 
    \[
    \deg_{\min}(\beta_{12}) = \deg_{\min}(\beta_{22}) = -a_1-\dots-a_k.
    \]
    The minimal degree of 
    \begin{align*}
    \beta_{11} + \beta_{12} &= (\beta'_{11}+\beta'_{12})(f_{a_k}+(-t)^{-a_k}) - \beta'_{11} t, \text{ and} \\
    \beta_{21} + \beta_{22} &= (\beta'_{21}+\beta'_{22})(f_{a_k}+(-t)^{-a_k}) - \beta'_{21} t 
    \end{align*}
    are both $-a_1 - \dots - a_k$ since $\deg_{\min}((\beta'_{11}+\beta'_{12}))= \deg_{\min}((\beta'_{21}+\beta'_{22})) = -a_1 -\dots -a_{k-1}$ and $\deg_{\min}(\beta'_{ 11})$, $\deg_{\min}(\beta'_{21}) > -a_1 - \dots - a_{k-1}$ if $\beta'_{21} \neq 0$. By induction hypothesis, we have $\deg_{\min}(t\beta'_{11}) > -a_1-\dots-a_{k-1}+1$ and therefore
    \[
    \deg_{\min}(\beta_{11}) = \deg_{\min}((\beta'_{11}+\beta'_{12})f_{a_k}) = -a_1-\dots-a_k+1 > \deg_{\min}(\beta_{22}).
    \]
    Similarly, since $\deg_{\min}(t\beta'_{21}) > -a_1-\dots-a_{k-1}+1$ we have 
    \[
    \deg_{\min}(\beta_{21}) = \deg_{\min}((\beta'_{21}+\beta'_{22})f_{a_k}) = -a_1-\dots-a_k+1 > \deg_{\min}(\beta_{22}).
    \]
    The smallest term of $\tr(\rho(\beta))$ is the same as the smallest term of $\beta'_{22}(-t)^{-a_k}$ which is $(-t)^{-a_1-\dots-a_k}$. This implies that $\tr(\rho(\beta))$ is positive in $(\mathbb{E},Q)$ if and only if $-a_1-\dots-a_k$ is even. Finally, by comparing the minimal degree of the terms in the discriminant of $\rho(\beta)$, we have
    \[
    \deg_{\min}(\tr(\rho(\beta)))^2 = -2(a_1+\dots+a_k) < -a_1-\dots-a_k + k = \deg_{\min}(\det(\rho(\beta)))
    \]
    where $a_i \geq 0$ for all $1\leq i \leq k$. Therefore, the discriminant of $\rho(\beta)$ is positive in $(\mathbb{E},Q)$. 
\end{proof}

An immediate corollary of \cref{prop:PositiveDiscriminant} is the following:

\begin{cor}
    \label{cor:EigenvaluesOfpABraids}
    Let $\beta$ be a braid in $B_3$ that is conjugate to $\sigma_1\sigma_2^{-a_1}\dots\sigma_1\sigma_2^{-a_k}\Delta^{2d}$, where $a_i \geq 0$ with at least one $a_i \neq 0$. Then both eigenvalues of $\rho(\beta)$ are in $\mathbb{E}$. Furthermore, the eigenvalues of $\rho(\beta)$
    \begin{enumerate}
        \item are both positive in $(E,\mathbb{Q})$ if $k$ and $a_1+\dots+a_k$ are both even,
        \item are both negative if $k$ and $a_1-\dots-a_k$ are both odd, and
        \item have opposite signs if $k -a_1-\dots-a_k$ is odd.
    \end{enumerate}
\end{cor}

\begin{proof}
    Let $\beta$ be a braid in $B_3$ that is conjugate to $\sigma_1\sigma_2^{-a_1}\dots\sigma_1\sigma_2^{-a_k}\Delta^{2d}$, where $a_i \geq 0$ with at least one $a_i \neq 0$. The eigenvalues of $\rho(\beta)$ are solutions to 
    \[
    \lambda^2 - \tr(\rho(\beta))\lambda + \det(\rho(\beta)).
    \]
    By \cref{prop:PositiveDiscriminant}, $\tr(\rho(\beta))^2 - 4\det(\rho(\beta))$ is positive in $(\mathbb{E},Q)$.  Since $\mathbb{E}$ is a real closed field,  \cref{thm:ExistenceOfRealClosureAndProperties} implies that $\tr(\rho(\beta))^2 - 4\det(\rho(\beta))$ is a square in $\mathbb{E}$. It follows that the eigenvalues of $\rho(\beta)$ are both in $\mathbb{E}$. If $k$ and $a_1+\dots+a_k$ are both even, then $\tr(\rho(\beta))$ and $\det(\rho(\beta))$ are both positive in $(\mathbb{E},Q)$ which implies that both eigenvalues are positive in $(\mathbb{E},Q)$. If $k$ and $a_1+\dots+a_k$ are both odd, then $\tr(\rho(\beta))$ is negative in $(\mathbb{E},Q)$ while $\det(\rho(\beta))$ is positive in $(\mathbb{E},Q)$. In this case, both eigenvalues of $\rho(\beta)$ are negative in $(\mathbb{E},Q)$. Finally, if $k-a_1-\dots-a_k$ is odd, then $\det(\rho(\beta))$ is negative in $(\mathbb{E},Q)$ which implies that the eigenvalues of $\rho(\beta)$ have different signs in $(\mathbb{E},Q).$
\end{proof}

\begin{proof}[Proof of \cref{thm:EvenEvenBraidsAreOP}]
    Let $\beta$ be a braid in $B_3$ that is conjugate to $\sigma_1\sigma_2^{-a_1}\dots\sigma_1\sigma_2^{-a_k}\Delta^{2d}$, where $a_i \geq 0$ with at least one $a_i \neq 0$. If both $k$ and $a_1+\dots+a_k$ are even, then the eigenvalues of $\rho(\beta)$ are both positive in $(\mathbb{E},Q)$ by \cref{cor:EigenvaluesOfpABraids}. By \cref{thm:OPViaBurauRep}, the braid $\beta$ is order-preserving.
\end{proof}

\begin{proof}[Proof of \cref{cor:SquareOfThreeBraidsAreOP}]
    Since $\beta$ is a non-periodic braid, $\beta$ is conjugate to a braid in the first two cases of \cref{thm:MurasugiClassification}. In the second case, the underlying permutation of $\beta$ is either a 2-cycle or the identity. In either case, the braid $\beta^2$ is a pure braid and is order-preserving by \cite[Prop 4.6]{KinRolfsen18}. In the first case of \cref{thm:MurasugiClassification}, we may assume that $d =0$ since the action of $\Delta^2$ in $\Out(F)$ is trivial. The braid $\beta^2$ is conjugate to a braid of the form $\sigma_1\sigma_2^{-a_1}\dots\sigma_1\sigma_2^{-a_k}$, where $a_i \geq 0$ with at least one $a_i \neq 0$ and where $k$ and $a_1+\dots+a_k$ are both even. By \cref{cor:EigenvaluesOfpABraids}, both eigenvalues of the matrix $\rho(\beta^2)$ are positive in $(\mathbb{E},Q)$. \cref{thm:OPViaBurauRep} implies that $\beta^2$ is an order-preserving braid. 
\end{proof}

\begin{remark}
    It follows from the work of Perron and Rolfsen that any pure braid must be order-preserving. The fact that pure braid is order-preserving has also been observed by Kin and Rolfsen \cite[Prop 4.6]{KinRolfsen18}. Therefore, $\beta^6$ is always an order-preserving braid for any $\beta \in B_3$. The key consequence of \cref{cor:SquareOfThreeBraidsAreOP} is that, in fact, there are a lot of order-preserving 3-braids whose underlying permutation is a 3-cycle. In fact, not a single example of order-preserving $n$-braid of full cycle type is known prior to our paper for any $n$. See \cite[Question 6.10]{KinRolfsen18} and the remark in \cite{CaiClayRolfsen} after Theorem 4.2.
\end{remark}

\begin{remark}
    There are a lot of known examples of non-order-preserving 3-braids whose underlying permutation is a 3-cycle. Most recently, Schreich, Turner, and the first author showed that the family $\sigma_1\sigma_2^{-2k-1}$ is not order-preserving for any $k \in \mathbb{Z}$. The complement of the braided link $(\sigma_1\sigma_2^{-1})^2$ can be realized as a 4-fold irregular cover of the complement of the figure 8-knot. Since the fundamental group of the complement of the figure-8 knot is known to be bi-orderable, via \cite[Theorem 1.1]{PerronRolfsen03} for example, this gives an independent way to see that $(\sigma_1\sigma_2^{-1})^2$ is an order-preserving braid. 
\end{remark}

\subsection{Sporadic examples of order-preserving $n$-braids}
\label{subsec:SporadicExamples}

In view of \cite[Question 6.10]{KinRolfsen18}, we present some sporadic examples of order-preserving $n$-braids for $n \geq 4$ with emphasis on braids whose underlying permutation has one cycle. In particular, we will give some examples of braids that satisfy the hypothesis of \cref{thm:OPViaBurauRep}. For convenience, we say that a braid $\beta \in B_n$ has {\bf positive Burau eigenvalues} if the matrix $\rho(\beta)$ from \cref{eq:ReducedBurauRep} has all $n-1$ positive eigenvalues in $\mathbb{E}$ counted with multiplicity. We will also say that $\beta$ has {\bf one cycle} to mean that the underlying permutation of $\beta$ has one cycle. Before presenting our examples, we make the following observation:

\begin{lem}
    \label{lem:OnlyWorksForOddStrandBraids}
    If $\beta \in B_{2n}$ has one cycle, then $\beta$ does not have positive Burau eigenvalues.      
\end{lem}

\begin{proof}
    Since the number of strands is even, the underlying permutation of $\beta$ is an odd permutation. This implies that $\beta$ can only be written as a product of an odd number of $\sigma_i$'s generators. The determinant $\det(\rho(\beta)) = -t^m$ for some $m \in \mathbb{Z}$ and is negative in $\mathbb{E}$. Therefore, $\beta$ cannot have positive Burau eigenvalues.    
\end{proof}

In particular, \cref{lem:OnlyWorksForOddStrandBraids} says that we can only use \cref{thm:OPViaBurauRep} to produce order-preserving braids of one-cycle type when the number of strands is odd. The following are our examples:

\begin{prop}
    The braids 
    \begin{itemize}
        \item $\sigma_4^{-3}\sigma_3^{-3}\sigma_2^3\sigma_1^3$ in $B_5$,
        \item $(\sigma_6^{-3}\sigma_5^{-3}\sigma_4^{-3}\sigma_3^{3}\sigma_2^{3}\sigma_1^{3})^2$ in $B_7$,
        \item $\sigma_8^{-3}\dots\sigma_5^{-3}\sigma_4^{3}\dots\sigma_1^{3}$ in $B_9$
    \end{itemize}  
    have positive Burau eigenvalues and are order-preserving.
\end{prop}

\begin{proof}
    Since $\mathbb{E}$ is a real closed field, polynomials in one variable over $\mathbb{E}$ satisfies intermediate value theorem \cite[Chapter XI, \S 2, Theorem 2.5]{LangAlgebra}. In other words, we can count the number of roots in $\mathbb{E}$ of a polynomial in one variable over $\mathbb{E}$ by observing the number of sign changes of the polynomial. 

    For brevity, let $\beta_5 = \sigma_4^{-3}\sigma_3^{-3}\sigma_2^3\sigma_1^3$, $\beta_7 = (\sigma_6^{-3}\sigma_5^{-3}\sigma_4^{-3}\sigma_3^{3}\sigma_2^{3}\sigma_1^{3})^2$ and $\beta_9 = \sigma_8^{-3}\dots\sigma_5^{-3}\sigma_4^{3}\dots\sigma_1^{3}$.
    Denote the characteristic polynomials of $\rho(\beta_5)$, $\rho(\beta_7)$ and $\rho(\beta_9)$ by $\chi_5(\lambda)$, $\chi_7(\lambda)$ and $\chi_9(\lambda)$. A straightforward calculation shows that $\chi_5$, $\chi_7$ and $\chi_9$ are all symmetric polynomials of degree 4,6 and 8. The roots of these polynomials are either $\pm 1$ or pairs of reciprocal numbers. To show that $\chi_5$, $\chi_7$ and $\chi_9$ have all positive roots counted with multiplicity in $\mathbb{E}$, we just need to show that $\chi_5$, $\chi_7$ and $\chi_9$ have half of the number of roots in the interval $(0,1) \subset \mathbb{E}$. 
     
     Recall that $t^n < 1$ for all $n>0$. We check that
     \begin{itemize}
         \item  $\chi_5(1) = t^{-6} + O(t^{-5})$, $\chi_5(t^2) = -t^{-3} + O(t^{-2})$ and $\chi_5(t^5) = 2t + O(t^{3})$
         \item  $\chi_7(1) = -t^{-18} + O(t^{-17})$, $\chi_7(t^2) = 3t^{-12} + O(t^{-11})$,  $\chi_7(t^6) = -4t^{-3} + O(t^{-2})$ and $\chi_7(t^{11}) = 1 + O(t)$
         \item $\chi_9(1) = t^{-12} +O(t^{-11})$, $\chi_9(t) = -t^{-8} +O(t^{-7})$, $\chi_9(t^2) = t^{-6} +O(t^{-5})$, $\chi_9(t^5) = -1 +O(t)$, and $\chi_9(t^6) = 1 +O(t)$.
     \end{itemize}
     Thus, all three polynomials change sign 2, 3, and 4 times in the interval $(0,1)$ respectively so all three polynomials $\chi_5$, $\chi_7$ and $\chi_9$ have all positive roots in $\mathbb{E}$ counted with multiplicity. The braids $\beta_5$, $\beta_7$ and $\beta_9$ have positive Burau eigenvalues and are order-preserving.

\end{proof}

\section{Final remarks}
\label{subsec:ConjecturesQuestions}

\cref{thm:OPViaBurauRep} has a direct generalization to outer automorphisms of free groups or bi-orderable surface groups that leave invariant a cohomology class. Indeed, let $\varphi$ be an outer automorphism of a free group or a bi-orderable surface group $G$ such that $\mu = \mu \circ \varphi$ for some cohomology class $\mu: G \to \mathbb{Z}$. In either case, the kernel of $\mu$, $K = \ker(\mu)$, is an infinitely generated free abelian group. The $\mathbb{Z}$-action associated to $\mu$ on $H_1(K;\mathbb{Z})$ still gives $H_1(K;\mathbb{Z})$ the structure of a finitely-generated free $\mathbb{Z}[t^{\pm 1}]$-module. Therefore, if the induced action of $\varphi$ on $H_1(K;\mathbb{Z})$ as a $\mathbb{Z}[t^{\pm 1}]$-module has all positive eigenvalues, then $\varphi$ preserves a bi-order on $G$.  

An initial motivation of the paper is to produce examples of order-preserving braids whose underlying permutation is a single cycle, addressing a question in \cite{KinRolfsen18}. To this end, we proved a general criterion that braids with positive Burau eigenvalues are order-preserving. We explore the extent to which a converse of \cref{thm:OPViaBurauRep} may hold.

There are several families of 3-braids that are known to be not order-preserving in the literature:
\begin{itemize}
    \item $\sigma_1 \in B_3$ \cite[Prop 4.4]{KinRolfsen18},
    \item $\sigma_1\sigma_2$ \cite[Theorem 4.10]{KinRolfsen18},
    \item $\sigma_1\sigma_2^{-2k-1}$ for $k=0$ \cite[Theorem 4.12]{KinRolfsen18} for all $k \in \mathbb{Z}$ \cite{JohnsonScherichTurner2024},
    \item $\sigma_1\sigma_2\sigma_1^{2k}$ for $k \geq 1$ \cite[Theorem 6.1]{KinRolfsen18},
    \item $(\sigma_1\sigma_2)^2\sigma_1^{2k}$ for $k \geq 1$ \cite[Theorem 6.3]{KinRolfsen18}.
\end{itemize}

We checked that if $\beta$ is a braid conjugate to a braid in one of the families listed above, then both eigenvalues of $\rho(\beta)$ are negative in $(\mathbb{E},Q)$ except in the following two cases. The braid $\beta$ could be conjugate to $\sigma_1$ which has one positive and one negative eigenvalue in $(\mathbb{E},Q)$. Alternatively, the braid $\beta$ could be conjugate to $\sigma_1\sigma_2$ which has two non-real eigenvalues. On the other hand, the periodic braid $\sigma_1\sigma_2\sigma_1$ is order-preserving \cite[Theorem 4.10]{KinRolfsen18}, while the eigenvalues of $\rho(\sigma_1\sigma_2\sigma_1)$ have opposite signs in $(\mathbb{E},Q)$. These calculations indicate that the converse of \cref{thm:OPViaBurauRep} is false. Nevertheless, we may expect the following partial converse to hold as an analog of the result by Clay and Rolfsen \cite[Theorem 3.4]{ClayRolfsen12}

\begin{conj}\label{conj:ObstructionToOPBraidViaBurau}
    Let $\rho:\beta \to \GL_{n-1}(\mathbb{Z}[t^{\pm 1}])$ be the (reduced) Burau representation. If $\beta \in B_n$ is an order-preserving braid, then $\rho(\beta)$ has at least one positive eigenvalue in $(\mathbb{E},Q)$.
\end{conj}
\noindent
If \cref{conj:ObstructionToOPBraidViaBurau} is true, then the conjecture and \cref{cor:EigenvaluesOfpABraids} would provide a more comprehensive understanding of when a 3-braid is order-preserving. 

As remarked in \cref{subsec:SporadicExamples}, braids with an even number of strands and of one-cycle type do not have positive Burau eigenvalues. Therefore, a refinement of \cite[Question 6.1]{KinRolfsen18} is the following

\begin{quest}
    Does there exist an order-preserving braid whose underlying permutation has one cycle in $B_{2n}$?
\end{quest}

\printbibliography
\end{document}